\documentclass[english,reqno]{amsart}

\usepackage{amssymb} %to produce blackboard bold characters(\supsetneq)
\usepackage{enumitem} %lay­out of the three ba­sic list en­vi­ron­ments
\usepackage{mathtools} %an extension package to amsmath + loading amsmath
\usepackage[table]{xcolor} %color of tables
\usepackage[all]{xy} %xymatrix commutative diagram
\usepackage{tikz} %pictures
\usepackage{indentfirst} %Indent first paragraph
\usepackage{babel} %lan­guages
\usepackage{setspace} %set­ting the spac­ing be­tween line
\usepackage{tikz-cd}

\usepackage[colorlinks,linkcolor=red,anchorcolor=green,citecolor=blue]{hyperref} %color of hyperlinks
\hypersetup{linktocpage = true} %color of hyperlinks of TOC

\usepackage{rotating} %Rotate an image, table or paragraph
\usepackage{ytableau} %young tableau
\usepackage{longtable} %long table
\newcolumntype{M}[1]{>{\centering\arraybackslash}m{#1}} %define dimension for long stable

\usepackage{mathpazo}
\usepackage{mathrsfs}
\DeclareFontFamily{OMS}{rsfs}{\skewchar\font'60}
\DeclareFontShape{OMS}{rsfs}{m}{n}{<-5>rsfs5 <5-7>rsfs7 <7->rsfs10 }{}
\DeclareSymbolFont{rsfs}{OMS}{rsfs}{m}{n}
\DeclareSymbolFontAlphabet{\scr}{rsfs}
\DeclareSymbolFontAlphabet{\scr}{rsfs}

\usepackage[T1]{fontenc}
\newcommand\cO{{\mathcal O}}

\newcommand\bbC{{\mathbb C}}

\newcommand\bbN{{\mathbb N}}
\newcommand\bbP{{\mathbb P}}
\newcommand\bbQ{{\mathbb Q}}

\newcommand\bbS{{\mathbb S}}
\newcommand\bbZ{{\mathbb Z}}

\newcommand\sE{{\mathscr E}}
\newcommand\sF{{\mathscr F}}
\newcommand\sG{{\mathscr G}}
\newcommand\sL{{\mathscr L}}
\newcommand\sM{{\mathscr M}}
\newcommand\sN{{\mathscr N}}
\newcommand\sO{{\mathscr O}}
\newcommand\sQ{{\mathscr Q}}

\newcommand{\codim}{{\rm codim}}

\newcommand{\Gr}{\rm{Gr}}

\DeclareMathOperator*{\pic}{Pic}

\DeclareMathOperator*{\red}{red}

\DeclareMathOperator*{\Exc}{Exc}

\newcommand{\NE}[1]{\ensuremath{\overline{\mbox{NE}}(#1)}}

\newcommand{\defeq}{{\vcentcolon=}}

\newtheorem{lemma1}{}[section]
\newenvironment{lemma}{\begin{lemma1}{\bf Lemma.}}{\end{lemma1}}
\newenvironment{example}{\begin{lemma1}{\bf Example.}\rm}{\end{lemma1}}
\newenvironment{thm}{\begin{lemma1}{\bf Theorem.}}{\end{lemma1}}
\newenvironment{prop}{\begin{lemma1}{\bf Proposition.}}{\end{lemma1}}

\newenvironment{remark}{\begin{lemma1}{\bf Remark.}\rm}{\end{lemma1}}

\newenvironment{defn-prop}{\begin{lemma1}{\bf Definition-Proposition.}}{\end{lemma1}}

\newenvironment{thm A}{{\bf Theorem A.}}{}
\newenvironment{thm B}{{\bf Theorem B.}}{}
\newenvironment{thm C}{{\bf Theorem C.}}{}
\newenvironment{thm D}{{\bf Theorem D.}}{}
\newenvironment{remark*}{{\bf Remark.}}{}
\newenvironment{example*}{{\bf Example.}}{}
\newenvironment{assumption*}{{\bf Assumption.}}{}

\newenvironment{conclusion*}{{\bf Conclusion.}}{}

\theoremstyle{definition}

\setlist[itemize]{leftmargin=*}
\setlist[enumerate]{leftmargin=*}

\numberwithin{equation}{section} %numbering of equations
\makeatletter

\ifnum\@ptsize=0  \fi \ifnum\@ptsize=2  \fi

\title{Fano manifolds containing a negative divisor isomorphic to a rational homogeneous space of Picard number one} 
\date{\today}

\subjclass[2010]{14J45, 14M17, 14E30}
\keywords{Fano manifolds, rational homogeneous spaces, extremal contraction, Mori theory}

\author{Jie Liu}

\address{Jie Liu, Institute of Mathematics, Academy of Mathematics and Systems Science, Chinese Academy of Sciences, Beijing, 100190, China}
\email{jliu@amss.ac.cn}

\begin{document}

\begin{abstract}
	Let $X$ be an $n$-dimensional complex Fano manifolds $(n\geq 3)$. Assume that $X$ contains a divisor $A$, which is isomorphic to a rational homogeneous space with Picard number one, such that the conormal bundle $\sN^*_{A/X}$ is ample over $A$. Building on the works of Tsukioka, Watanabe and Casagrande-Druel, we give a complete classification of such pairs $(X,A)$.
\end{abstract}

\maketitle

\tableofcontents

\vspace{-0.2cm}

\section{Introduction}

Throughout this paper, we work over the complex numbers. As an application of Mori theory, Mori and Mukai succeeded in classifying Fano threefolds by viewing extremal rays in details \cite{MoriMukai1981}. In general it is very difficult to describe the extremal rays and contractions of higher dimensional Fano manifolds. However, if there exists a "special" divisor on $X$, then it is possible to get various information about $X$. For example, Watanabe classified Fano manifolds $X$ containing an ample divisor isomorphic to a homogeneous space in \cite{Watanabe2008}.

\begin{thm}\cite[Theorem 1]{Watanabe2008}
	\label{Watanabe-Homogeneous}
	Let $X$ be a projective manifold of dimension $n\geq 3$ containing an ample divisor $A$ isomorphic to a rational homogeneous space. If $\rho(A)=1$, then the pair $(X,\sO_X(A))$ is isomorphic to one of the following:
	\begin{enumerate}
		\item $(\bbP^n,\sO_{\bbP^n}(1))$;
		
		\item $(\bbP^n,\sO_{\bbP^n}(2))$ and $n\geq 4$;
		
		\item $(Q^{n},\sO_{Q^n}(1))$ and $n\geq 4$;
		
		\item $(\Gr(2,2n),\sO(1))$ and $n\geq 2$, where $\Gr(2,2n)$ is the Grassmannian of $2$-dimensional subspaces in an $2n$-dimensional vector space and $\sO(1)$ is the ample generator of the Picard group of $\Gr(2,2n)$;
		
		\item $(E_6/P_1,\sO(1))$, where $E_6/P_1$ is the $27$-dimensional rational homogeneous space of type $E_6$ and $\sO(1)$ is the ample generator of the Picard group of $E_6/P_1$.
	\end{enumerate}
\end{thm}

On the other hand, Bonavero, Campana and Wi\'sniewski gave the classification of $n$-dimensional complex Fano manifolds $X$ $(n\geq 3)$ containing a divisor $A$ isomorphic to $\bbP^{n-1}$ with normal bundle $\sN_{A/X}\simeq \sO_{\bbP^{n-1}}(-1)$ in \cite[Theorem 1]{BonaveroCampanaWisniewski2002}. Some years latter, in \cite{Tsukioka2006}, Tsukioka generalized this result to the case where $\sN_{A/X}$ is isomorphic to $\sO_{\bbP^{n-1}}(-d)$ for some integer $d>0$. In particular, Tsukioka proved in \cite[Proposition 5]{Tsukioka2006} that if an $n$-dimensional Fano manifold $X$ $(n\geq 3)$ contains a prime divisor $A$ with $\rho(A)=1$, then $\rho(X)\leq 3$. In \cite[Lemma 3.1 and Theorem 3.8]{CasagrandeDruel2015}, Casagrande and Druel described in details the extremal contractions of such a pair $(X,A)$ and gave a general classification of such pairs in the extremal case $\rho(X)=3$. 

The main result of this note is to generalize the results of \cite{Tsukioka2006} and \cite{BonaveroCampanaWisniewski2002} to the case where $A$ is isomorphic to a rational homogeneous space with Picard number one.

\begin{thm}\label{Classification-Negative-Divisors}
	Let $X$ be a Fano manifold of dimension $n\geq 3$ containing a divisor $A$ isomorphic to a rational homogeneous space with Picard number one. Denote by $\sO_A(1)$ the ample generator of $\pic(A)$ and by $r$ the index of $A$. Assume that $\sN_{A/X}$ is isomorphic to $\sO_A(-d)$ for some integer $d>0$. Then $0<d<r$ and we are in one of the following cases.
	\begin{enumerate}
		\item  $\rho(X)=2$ and the pair $(X,A)$ is isomorphic to one of the following:
		\begin{enumerate}
			\item $X$ is isomorphic to $\bbP(\sO_A\oplus \sO_A(-d))$ and $A$ is a section with normal bundle $\sN_{A/X}\simeq \sO_A(-d)$;
			
			\item $X$ is obtained by blowing up one of the pairs $(X',A')$ listed in Theorem \ref{Watanabe-Homogeneous} along a smooth center $C\in \vert \sO_{A'}(d+s)\vert$, where $\sO_{A'}(1)$ is the ample generator of $\pic(A')$, $\sN_{A'/X'}\simeq \sO_{A'}(s)$ and $A$ is the strict transform of $A'$.
			
			{\color{red}\item $X$ is a smooth element in $|\sO_{\bbP^{n-1}}(\sE)(2)\otimes \pi^*\sO_{\bbP^{n-1}}(2)|$ and $A$ is isomorphic to a quadric hypersurface such that $X\cap F=A$, where $\sE$ is the vector bundle $\sO_{\bbP^{n-1}}(d)\oplus \sO_{\bbP^{n-1}}\oplus \sO_{\bbP^{n-1}}(-1)$, the map $\pi:\bbP(E)\rightarrow \bbP^{n-1}$ is the natural projection and the variety $F\subset \bbP(E)$ is the subbundle corresponding to the quotient $\sE\rightarrow \sO_{\bbP^{n-1}}\oplus \sO_{\bbP^{n-1}}(-1)$.}
		\end{enumerate}
	
	\item $\rho(X)=3$ and $X$ is obtained by blowing up a Fano manifold $Y$ along a smooth center $C\in \vert \sO_{A_Y}(d+s)\vert$ such that $-d<s<r$, where $Y$ is isomorphic to $\bbP(\sO_A\oplus\sO_A(s))$, $A_Y$ is a section with normal bundle $\sN_{A_Y/Y}\simeq \sO_{A}(s)$, $\sO_{A_Y}(1)$ is the ample generator of $\pic(A_Y)$ and $A$ is the strict transform of $A_Y$.
	\end{enumerate}	
\end{thm}

{\color{red}
\begin{remark}
	In the published version of this paper [Internat. J. Math., 2020, 31, 2050066, 14], the case (1.3) is missed in the statement. The mistake appears in the proof of Lemma \ref{Section-Criterion}, which is false in general. See Appendix \ref{corrigendum} for the correction.
\end{remark}
}

\subsection*{Acknowledgements} I would like to thank Baohua Fu for patiently answering my numerous questions. This work is supported by China Postdoctoral Science Foundation (2019M650873). I want to thank the referee for pointing out some inaccuracies in the first version.

\section{Examples and inextendability criterion}

\subsection{Examples}

In this subsection, we collect some examples of Fano manifolds.

\begin{example}\label{Example1}
	Let $(Y,D)$ be a pair where $Y$ is a Fano manifold with $\pic(Y)\simeq \bbZ\sO_Y(1)$ for some $\sO_Y(1)$ ample and $D\in \vert\sO_Y(s)\vert$ is a smooth member with $s>0$. Assume that the restriction $\pic(Y)\rightarrow\pic(D)$ is surjective. Denote by $\sO_D(1)=\sO_Y(1)\vert_D$ the ample generator of $\pic(D)$. Suppose moreover that $D$ is a Fano manifold with index $r$. For a given positive integer $d$, we choose a smooth member $B\in\vert \sO_D(d+s)\vert$ and denote by $\pi\colon X\rightarrow Y$ the blow-up of $Y$ along $B$. Let $\widetilde{D}$ be the strict transform of $D$ in $X$. Then $\widetilde{D}$ is isomorphic to $D$.
\end{example}

\begin{lemma}\label{Claim1}
	In Example \ref{Example1}, $X$ is a Fano manifold if and only if $d<r$.
\end{lemma}

\begin{proof}
	Firstly we show $\sN_{\widetilde{D}/X}\simeq \sO_{\widetilde{D}}(-d)$. Denote by $E$ the exceptional divisor of $\pi$. By assumption, we have
	\[
	\sO_D(s)\simeq\sO_Y(D)\vert_D\simeq \sO_X(\pi^*D)\vert_{\widetilde{D}}\simeq \sO_X(\widetilde{D}+E)\vert_{\widetilde{D}}
	\]
	and
	\[
	\sO_X(E)\vert_{\widetilde{D}}\simeq \sO_{D}(B)\simeq \sO_{D}(d+s).
	\]
	It follows immediately that $\sN_{\widetilde{D}/X}\simeq \sO_X(\widetilde{D})\vert_{\widetilde{D}}\simeq \sO_{\widetilde{D}}(-d)$.
	
	Next we show that $-K_X$ is ample. As $\rho(X)=2$, we can assume that $R_1$ and $R_2$ are the generators of $\NE{X}$. Moreover, by the construction, we may assume that $R_1$ is generated by the curves contained in the fibers of the blow-up $\pi\colon X\rightarrow Y$. We claim that $R_2$ is generated by curves contained in $\widetilde{D}$. Indeed, to see this, it suffices to find a non-trivial nef $\bbQ$-divisor $A$ such that $A\vert_{\widetilde{D}}\equiv 0$. We set 
	\[
	A=\widetilde{D}+\frac{d}{s}\pi^*D=\frac{d+s}{s}\widetilde{D}+\frac{d}{s}E.
	\]
	Then we have $A\vert_{\widetilde{D}}\equiv 0$. Moreover, since $\pi^*D$ is nef and big and $\widetilde{D}$ is effective, we obtain $A\cdot C\geq 0$ for any irreducible curve $C$ not contained in $\widetilde{D}$. In particular, $A$ is nef and consequently $R_2$ is generated by curves contained in $\widetilde{D}$.
	
	Finally, note that $\pi$ is an extremal contraction, we have $-K_X\cdot R_1>0$. Moreover, by the adjunction formula, we get
	\[
	\sO_X(-K_X)\vert_{\widetilde{D}}\simeq \sO_{\widetilde{D}}(-K_{\widetilde{D}})\otimes\sO_X(\widetilde{D})\vert_{\widetilde{D}}\simeq \sO_{\widetilde{D}}(r-d).
	\]
	Hence, by Kleiman's criterion, $-K_X$ is ample if and only if $d<r$.
\end{proof}

\begin{example}\label{Example2}
	Fix integers $n$ and $d$ such that $n\geq 3$. Let $Y$ be a Fano manifold of dimension $n-1$, with $\rho(Y)=1$ and index $r$. Let $\sO_Y(1)$ be the ample generator of $\pic(Y)$. Set $X\defeq \bbP(\sO_Y\oplus\sO_Y(d))$, and denote by $\pi\colon X\rightarrow Y$ the natural projection.
\end{example}

\begin{lemma}\label{Claim2}
	In Example \ref{Example2}, $X$ is a Fano manifold if and only if $-r<d<r$.
\end{lemma}

\begin{proof}
	As $\rho(X)=2$, we shall denote by $R_1$ and $R_2$ the generators of $\NE{X}$. Moreover, as $\pi$ is an extremal contraction, we may assume that $R_1$ is generated by the curves contracted by $\pi$. On the other hand, since $X$ is isomorphic to $\bbP(\sO_Y\oplus\sO_Y(-d))$, without loss of generality, we shall assume that $d\geq 0$. If $d=0$, then $X$ is isomorphic to the product $Y\times \bbP^1$, and it is clear that $X$ is a Fano manifold. Thus we may assume that $d>0$. 
	
	Let $E$ be the section of $\pi$ with normal bundle $\sN_{E/X}\simeq \sO_{E}(-d)$, and let $E'$ be a section of $\pi$ with normal bundle $\sN_{E'/X}\simeq \sO_{E'}(d)$. Then $E$ is disjoint from $E'$. In particular, we have $E'\vert_E\equiv 0$. Moreover, as $d>0$, $E'$ is nef. Therefore, $R_2$ is generated by the curves contained in $E$.
	
	As $\pi$ is an extremal contraction, we have $-K_X\cdot R_1>0$. On the other hand, by the adjunction formula, we get
	\[
	\sO_X(-K_X)\vert_E\simeq \sO_{E}(-K_E)\otimes \sO_{X}(E)\vert_E\simeq \sO_E(r-d).
	\]
	Hence, by Kleiman's criterion, $-K_X$ is ample if and only if $d<r$ .
\end{proof}

\begin{example}\label{Example3}
	Let $X\simeq \bbP(\sO_Y\oplus\sO_Y(d))$ be a Fano manifold as in Example \ref{Example2}, and let $E$ be a section of $\pi\colon X\rightarrow Y$ with normal bundle $\sN_{E/X}\simeq \sO_E(d)$. Suppose that $B$ is a smooth hypersurface of $E$ such that $\sO_E(B)\simeq \sO_{E}(d')$ for some $d'>0$. Denote by $\sigma\colon W\rightarrow X$ the blow-up of $X$ along $B$. 
\end{example}

\begin{lemma}\label{Claim3}
	In Example \ref{Example3}, $W$ is a Fano manifold if and only if $0<d'<r+d$, where $r$ is the index of $Y$.
\end{lemma}

\begin{remark}
	This result is proved in the case $d\geq 0$ by Casagrande and Druel in \cite[Example 3.4 and Lemma 3.5]{CasagrandeDruel2015} and a slight modification of the argument can be applied to the case $d<0$. We include a proof for the reader's convenience.
\end{remark}

\begin{proof}
	If $d\geq 0$, it is proved in \cite[Lemma 3.5 and Remark 3.6]{CasagrandeDruel2015}. Thus we may assume that $d<0$. Denote by $B_X$ the divisor $\pi^{-1}(\pi(B))$ and denote by $G$ the exceptional divisor of $\sigma\colon W\rightarrow X$. Let $E_W$ (resp. $B_W$) be the strict transform of $E$ (resp. $B_X$) in $W$. Note that $B_W$ is isomorphic to $B$. Let $l_B\subset B_W$ be a fiber of the contraction $B_W\rightarrow B_X\rightarrow \pi(B)$. Then we obtain
	\begin{center}
		$B_W\cdot l_B=(\sigma^*B_X-G)\cdot l_B=-1$ and $-K_W\cdot l_B=(-\sigma^*K_X-G)\cdot l_B=1$.
	\end{center}
    Moreover, let $A$ be an ample divisor on $Y$. Then $\sigma^*\pi^*A\cdot l_B=0$. As a consequence, we can find an extremal ray $R_W$ of $\NE{W}$ such that 
    \begin{center}
    	$B_W\cdot R_W<0$ and $\sigma^*\pi^*A\cdot R_W=0$. 
    \end{center}
    Then one can easily see that $R_W$ is actually generated by $[l_B]$. Let $\widehat{\sigma}\colon W\rightarrow \widehat{X}$ be the associated contraction. Then both $\widehat{X}$ and $\widehat{\sigma}(B_W)$ are smooth and $\widehat{\sigma}\colon W\rightarrow \widehat{X}$ is the blow-up of $\widehat{X}$ along the codimension $2$ submanifold $\widehat{\sigma}(B_W)$ (see \cite[Theorem 1.2]{Wisniewski1991}).
    
    As $\rho(\widehat{X})=2$, $\NE{\widehat{X}}$ is generated by two rays $R_1$ and $R_2$. In the following, we give a detail description of these two rays. Denote by $E_W$ and $E_{\widehat{X}}$ the strict transforms of $E$ in $W$ and $\widehat{X}$, respectively. Then $E_W$ and $E_{\widehat{X}}$ are isomorphic to $E$. Moreover, as $E_W$ is disjoint from $B_W$, we get
    \[
    \sO_{\widehat{X}}(E_{\widehat{X}})\vert_{E_{\widehat{X}}}\simeq \sO_{W}(E_W)\vert_{E_W}\simeq \sO_W(\sigma^*E-G)\vert_{E_W}\simeq \sO_E(d-d')
    \]
    and
    \[
    \sO_{\widehat{X}}(-K_{\widehat{X}})\vert_{E_{\widehat{X}}}\simeq \sO_{W}(-K_W)\vert_{E_{W}}\simeq \sO_{W}(-\sigma^*K_X-G)\vert_{E_W}\simeq \sO_E(r+d-d').
    \]
    By assumption, we have $r+d-d'<0$ and $d-d'<0$. As consequence, the ray generated by curves contained in $E_{\widehat{X}}$ is an extremal ray of $\NE{\widehat{X}}$ and we will denote it by $R_1$. On the other hand, let $l$ be a general fiber of $\pi\colon X\rightarrow Y$. Then the birational map 
    \[
    \sigma^{-1}\circ \widehat{\sigma}\colon X\dashrightarrow \widehat{X}\]
    is an isomorphism in a neighborhood of $l$. Denote its image by $\widehat{l}$. Then we have
    \[-K_{\widehat{X}}\cdot \widehat{l}=-K_X\cdot l=2.\]
    On the other hand, let $G_{\widehat{X}}$ be the strict transform of $G$ in $\widehat{X}$. Then we get
    \[
    \sO_{\widehat{X}}(G_{\widehat{X}})\vert_{G_{\widehat{X}}}\simeq \sO_W(\widehat{\sigma}^*G_{\widehat{X}})\vert_{G}\simeq \sO_W(G+B_W)\vert_{G}\simeq \sO_W(\sigma^*B_X)\vert_G.
    \]
    As a consequence, $\sO_{\widehat{X}}(G_{\widehat{X}})\vert_{\widehat{G}}$ is nef and so is $\sO_{\widehat{X}}(G_{\widehat{X}})$. Since $G_{\widehat{X}}$ is disjoint from $\widehat{l}$, we have $G_{\widehat{X}}\cdot\widehat{l}=0$. Hence, $R_2$ is generated by $[\widehat{l}]$ and it is an extremal ray of $\NE{\widehat{X}}$. In particular, $\widehat{X}$ is a Fano manifold. 
    
    Let $\widehat{\pi}$ be the extremal contraction associated to $R_2$. Then $\widehat{\pi}$ is of fiber type since the strict transform of every general fiber of $\pi\colon X\rightarrow Y$ is contracted by $\widehat{\pi}$. Moreover, as $E_{\widehat{X}}\cdot \widehat{l}=1$, it follows that $\widehat{\pi}(\widehat{X})=\widehat{\pi}(E_{\widehat{X}})\simeq Y$ and $\widehat{\pi}$ is a $\bbP^1$-bundle (see \cite[Lemma 2.12]{Fujita1987}). In particular, we have the following factorization:
    \[\xymatrix{
    	&   W\ar[dl]_{\widehat{\sigma}}\ar[dr]^{\sigma}\ar[dd]^{\varphi} &       \\
    	\widehat{X}\ar[dr]_{\widehat{\pi}} &                & X\ar[dl]^{\pi}\\
    	& Y &    
    }\]
	Thanks to \cite[Lemma 3.9]{CasagrandeDruel2015}, $\widehat{X}$ is isomorphic to $\bbP(\sO_Y\oplus\sO_Y(d'-d))$. Let $E'$ be a section of $\pi\colon X\rightarrow Y$ with $\sN_{E'/X}\simeq \sO_{E'}(-d)$. Then one can see that the strict transform $E'_{\widehat{X}}$ of $E'$ in $\widehat{X}$ is a section of $\widehat{\pi}\colon \widehat{X}\rightarrow Y$ containing $\widehat{\sigma}(B_W)$ such that 
	\[
	\sO_{E'_{\widehat{X}}}(\widehat{\sigma}(B_W))\simeq \sO_{X}(B_X)\vert_{E'}\simeq \sO_{E'}(d').
	\]
	Then applying \cite[Lemma 3.5 and Remark 3.6]{CasagrandeDruel2015} to $\widehat{X}$ shows that $W$ is a Fano manifold if and only if $d'-d<r$.
\end{proof}

\begin{remark}
	In Theorem \ref{Classification-Negative-Divisors}, the restriction $\sO_A(-K_X)\simeq \sO_A(r-d)$ is ample, thus we have always $0<d<r$. In particular, the examples given above show that the projective manifolds provided in Theorem \ref{Classification-Negative-Divisors} are indeed Fano manifolds.
\end{remark}

\subsection{Projective extension}

Recall that an irreducible non-degenerate smooth projective variety $X\subset\bbP^N$ is called \emph{projectively extendable} if there exists a projective variety $X'\subset\bbP^{N+1}$ and a hyperplane $H\subset\bbP^{N+1}$ such that $H$ intersects $X'$ transversely, $H\cap X'=X$ and $X'$ is not a cone. In this case, $X'$ is called a \emph{projective extension} of $X$. The following inextendability criterion due to Zak is very useful (see also \cite[Theorem 0.1 and Corollary 1]{Lprimevovsky1992}).

\begin{thm}\cite[Corollary 3]{Zak1991}\label{inextendability}
	If $X\subset\bbP^N$ is an irreducible, non-degenerate, smooth projective variety such that $h^1(X,T_X(-1))=0$, then either $X$ is a twisted cubic curve or a quadric or $X$ is inextendable.
\end{thm}

Let $X$ be a projective variety. Recall that a line bundle $\sL$ over $X$ is called \emph{simply generated} if the graded algebra 
\[R(X,\sL)\defeq\bigoplus_{m\geq 0} H^0(X,\sL^{\otimes m})\]
is generated by $H^0(X,\sL)$ as a $\bbC$-algebra. Moreover,  {\color{red}a line bundle $\sL$ is very ample if $\sL$ is ample and simply generated.} Using this notion we have the following useful very ampleness criterion.

\begin{prop}\label{Ampleness-Criterion}
	Let $X$ be a normal projective and let $\sL$ be an ample line bundle on $X$. Suppose that $D\in \vert \sL\vert$ is a member which is irreducible and reduced as a subscheme of $X$. If $h^1(X,\sO_X)=0$ and {\color{red}$\sL\vert_D$ is simply generated}, then $\sL$ is very ample.
\end{prop}

\begin{proof}
	It suffices to prove that $\sL$ is simply generated. As $h^1(X,\sO_X)=0$, the restriction map $H^0(X,\sL)\rightarrow H^0(D,\sL\vert_D)$ is surjective. Since $\sL\vert_D$ is simply generated, then \cite[Chapter I, Corollary 2.5]{Fujita1990} says that $\sL$ is itself simply generated.
\end{proof}

As a consequence of Proposition \ref{Ampleness-Criterion}, one can easily derive the following variant of Theorem \ref{inextendability}.

\begin{prop}\label{Quadric-Extension}
	Let $X$ be a normal projective variety of dimension $n\geq 3$, and let $\sL$ be an ample line bundle over $X$. Suppose that $h^1(X,\sO_X)=0$ and $Y\in \vert \sL\vert$ is a scheme-theoretically smooth member with {\color{red}$\sL\vert_Y$ simply generated}. If $h^1(Y,T_Y\otimes \sL^*\vert_Y)=0$, then one of the following statements holds.
	\begin{enumerate}
		\item The map $\Phi$ defined by the complete linear system $\vert \sL\vert$ is an embedding which sends $X$ to a cone over $\Phi(Y)$.
		
		\item The pair $(Y,\sL\vert_Y)$ is isomorphic to $(Q^{n-1}, \sO_{Q^{n-1}}(1))$, where $Q^{n-1}$ is a quadric hypersurface of dimension $n-1$.
	\end{enumerate}
    Furthermore, suppose in addition that $X$ is smooth, then the pair $(X,\sL)$ is isomorphic to either $(\bbP^{n},\sO_{\bbP^{n}}(1))$ or $(Q^{n},\sO_{Q^n}(1))$.
\end{prop}

\begin{proof}
	Since $Y$ is an ample divisor on $X$ and $n\geq 3$, $Y$ is connected. It follows that $Y$ is irreducible for $Y$ being smooth. Then, by Proposition \ref{Ampleness-Criterion}, $\sL$ is a very ample line bundle over $X$. 
	
	Denote by $\Phi$ the embedding defined by the complete linear system $\vert \sL\vert$. As $h^1(X,\sO_X)=0$, we have $h^0(X,\sL)=h^0(Y,\sL\vert_Y)+1$. In particular, as $Y$ is smooth, there exists a hyperplane $H$ of $\bbP(H^0(X,\sL))$ such that $H$ intersects $\Phi(X)$ transversely and $\Phi(Y)=H\cap \Phi(X)$. Therefore, by the definition of projective extension,
	either $\Phi(X)$ is a cone over $\Phi(Y)$ or $\Phi(X)$ is a projective extension of $\Phi(Y)$.
	
	According to Theorem \ref{inextendability}, if $\Phi(X)$ is a projective extension of $\Phi(Y)$, then $\Phi(Y)$ is a quadric hypersurface in $\bbP(H^0(Y,\sL\vert_Y))\simeq \bbP^n$ because $h^1(Y,T_Y\otimes \sL^*\vert_Y)=0$. 
	
	Suppose now that $X$ is smooth. If case (1) holds, then $\Phi(X)$ is smooth if and only if $\Phi(Y)$ is a projective space. In particular, the pair $(Y,\sL\vert_Y)$ is isomorphic to $(\bbP^{n-1},\sO_{\bbP^{n-1}}(1))$ and $(X,\sL)$ must be isomorphic to $(\bbP^n,\sO_{\bbP^n}(1))$. If case (2) holds, then it is well known that $(X,\sL)$ is isomorphic to $(Q^n,\sO_{Q^n}(1))$ (see for instance Theorem \ref{Watanabe-Homogeneous}).
\end{proof}

\section{Proof of Theorem \ref{Classification-Negative-Divisors}}

This section is devoted to prove Theorem \ref{Classification-Negative-Divisors}. Since $A$ is a negative divisor on $X$, we get $\rho(X)\geq 2$. On the other hand, as mentioned in the introduction, according to \cite[Proposition 5]{Tsukioka2006}, we also have $\rho(X)\leq 3$. 

\subsection{Case $\rho(X)=2$}

The proof of the following lemmas can be adapted from that of \cite[Lemma 1 and Lemma 2]{Tsukioka2006}, and follows some strategies used in \cite{BonaveroCampanaWisniewski2002}.

\begin{lemma}\label{Smoothness-Criterion}
	Let $\pi\colon X\rightarrow Y$ be the blow-up of a projective manifold $Y$ along an irreducible smooth center $C$ of codimension $2$. Suppose that $A\subset X$ is a smooth irreducible hypersurface such that $\pic(A)\simeq\bbZ\sO_A(1)$ and there exists a birational morphism $\varphi\colon X\rightarrow Y_0$ onto a projective variety sending $A$ to a point. Then the restriction $\pi\vert_A\colon A\rightarrow \pi(A)$ is an isomorphism.
\end{lemma}

\begin{proof}
	Denote by $E=\pi^{-1}(C)$ the exceptional divisor of $\pi$. If $A$ is disjoint from $E$, it is clear that $\pi\vert_A\colon A\rightarrow \pi(A)$ is an isomorphism. Now we shall asume that $E\cap A$ is not empty. Set $W\defeq (A\cap E)_{\red}$. Then the restriction $\varphi\vert_E\colon E\rightarrow \varphi(E)$ sends $W$ to a point. By \cite[Lemma 3.9]{CasagrandeDruel2015}, $W$ is a section of the $\bbP^1$-bundle $\pi\vert_E\colon E\rightarrow C$ with conormal bundle $\sN^*_{W/E}$ ample. On the other hand, as $\pic(A)\simeq\bbZ\sO_A(1)$ and $W$ is effective, the line bundle $\sN_{W/A}\simeq \sO_A(W)$ is ample. In particular, $\sN_{W/A}$ is different from $\sN_{W/E}$ and consequently $W$ is generically reduced. As $W$ is Cohen-Macaulay, $W$ is actually reduced. In particular, we have $W=A\cap E$. Hence, the restriction map $\pi\vert_A\colon A\rightarrow \pi(A)$ is an isomorphism. 
\end{proof}

\begin{lemma}\label{Section-Criterion}
	Let $X$ be a Fano manifold of dimension $n\geq 3$ and with $\rho(X)=2$, and let $A$ be a smooth Fano hypersurface of $X$ such that $\pic(A)\simeq \bbZ\sO_A(1)$ for some ample line bundle $\sO_A(1)$ and $\sN_{A/X}\simeq\sO_A(-d)$ for some $d>0$. Assume furthermore that $(A,\sO_A(1))$ is covered by lines, i.e. for any point $x\in A$, there exists a rational curve $C$ passing through $x$ such that $c_1(\sO_A(1))\cdot C=1$. If $X$ admits an extremal contraction $f\colon X\rightarrow \bbP^{n-1}$, which is a conic bundle, such that $f$ is finite over $A$, then $A$ is a section of $f$. In particular, $f$ is a $\bbP^1$-bundle and $A$ is isomorphic to $\bbP^{n-1}$.
\end{lemma}

{\color{red}
	\begin{remark}
		This statement is false in general. See Lemma \ref{Section-Criterion-appendix} fo the correct statement. The mistake appears in the computation of the value of $x$, which should be $e/(2d+(r-d)e)$, not $e/(d-r)$.
	\end{remark}
}

\begin{proof}
	Denote by $r$ the index of $A$, i.e., $\sO_A(-K_A)\simeq \sO_A(r)$. As $X$ is Fano, the line bundle $\sO_{A}(-K_X)\simeq \sO_{A}(r-d)$ is ample. We get $r>d$. Since $A$ is not nef and $X$ is Fano, there exists an extremal ray $R$ of $\NE{X}$ such that $A\cdot R<0$. Let $\pi\colon X\rightarrow Y$ be the associated contraction. Then $\Exc(\pi)\subset A$ as $A\cdot R<0$. On the other hand, every curve contained in $A$ has class in $R$ since $\rho(A)=1$ and $\sO_X(A)\vert_A\simeq \sN_{A/X}$ is negative. This implies that $A=\Exc(\pi)$ and that $\pi(A)$ is a point. By adjunction, we have
	\[K_{X}\sim_{\bbQ}\pi^*K_Y+\frac{r-d}{d} A.\]
	Denote $f^*\sO_{\bbP^{n-1}}(1)$ by $H$. Since $\rho(X)=2$ and the contraction map $f$ is supposed to be elementary, there exist $x,y\in \bbQ$ such that 
	\[H\equiv x\pi^*(-K_Y)-yA.\]
	Denote by $e$ the degree $\sO_A(1)^{n-1}$. Set $\alpha=(r-d)/d$ and $m\defeq(-K_Y)^n$. Then we get
	\begin{equation}\label{Eq 1}
		0=H^n=x^n m - y^n d^{n-1} e
	\end{equation}
	and
	\begin{equation}\label{Eq 2}
		2=(-K_X)\cdot H^{n-1}=x^{n-1}m - \alpha y^{n-1}d^{n-1}e.
	\end{equation}
	Set $l\defeq 2d^2/(yd)^{n-1}$. Then the proof of \cite[Lemma 1]{Tsukioka2006} can be applied verbatim in our case to obtain that $l$ is an integer and $2d^2=(yd)^{n-1}l$. 
	
	To prove $yd=1$, as in the proof \cite[Lemma 1]{Tsukioka2006}, it suffices to exclude the case $yd=2$ and $(n,d,l)\in \{(3,2,2),(5,4,2)\}$.  Indeed, if $(n,d,l)\in \{(3,2,2),(5,4,2)\}$, then we must have $r>d=n-1=\dim(A)$. By Kobayashi-Ochiai's theorem, then $A$ is isomorphic to $\bbP^{n-1}$, which is impossible as shown in the proof of \cite[Lemma 1]{Tsukioka2006}. Thus, we have $yd=1$ and as a consequence, we have
	\[A\cdot H^{n-1}=A\cdot(x\pi^*(-K_Y)-yA)^{n-1}=(yd)^{n-1}e=e.\]
	On the other hand, by \eqref{Eq 1} and \eqref{Eq 2}, we have
	\[x=\frac{y^nd^{n-1}e}{2+\alpha y^{n-1}d^{n-1}}=\frac{ye}{2+\alpha}=\frac{e}{r+d}.\]
	It yields
	\[0=H^{n}=(x\pi^*(-K_Y)-yA)\cdot H^{n-1}=\frac{e}{r+d}\pi^*(-K_Y)\cdot H^{n-1}-\frac{e}{d}\]
	and 
	\[2=(-K_X)\cdot H^{n-1}=(\pi^*(-K_Y)-\alpha A)\cdot H^{n-1}=\pi^*(-K_Y)\cdot H^{n-1}-\frac{(r-d)e}{d}.\]
	It follows that
	\[2=\frac{r+d}{d}-\frac{(r-d)e}{d}.\]
	Hence, as $r>d$, we have $e=1$. Consequently, $X$ is isomorphic to $\bbP(f_*\sO_X(A))$ and $A$ is a section of $f\colon X\rightarrow \bbP^{n-1}$ (see \cite[Lemma 2.12]{Fujita1987})
\end{proof}

Now we are in the position to prove the first part of Theorem \ref{Classification-Negative-Divisors}. Let us recall that rational homogeneous spaces are covered by lines, see for instance \cite[V, Theorem 1.15]{Kollar1996}.

\begin{proof}[Proof of Theorem \ref{Classification-Negative-Divisors} (1)]
	Denote by $R_1$ and $R_2$ the extremal rays of $\NE{X}$ and, without loss of generality, we shall assume $A\cdot R_1>0$ (cf. \cite[Lemma 3.1]{CasagrandeDruel2015}). Then we have the following diagram:
	
	
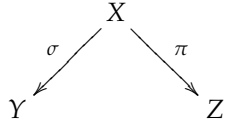
\[\xymatrix{
		&   X \ar[dl]_{\sigma}\ar[dr]^{\pi}    &  \\
		Y    &         & Z
	}\]
	where $\sigma$ (resp. $\pi$) is the extremal contraction corresponding to $R_1$ (resp. $R_2$). Since $A$ is not nef, by \cite[Remark 3.2]{CasagrandeDruel2015}, $A\cdot R_2<0$ and $\pi$ is a divisorial contraction sending $A$ to a point. Furthermore, by \cite[Lemma 3.1]{CasagrandeDruel2015}, the possibilities of $\sigma$ are as follows:
	\begin{enumerate}[label=(\alph*)]
		\item $\sigma$ is a conic bundle, finite on $A$, such that $Y$ is a Fano manifold;
		
		\item $\sigma$ is the blow-up of $Y$ along a smooth center $C$ of codimension $2$ and $Y$ is a Fano manifold.
	\end{enumerate}

    Suppose first that $\sigma$ is a conic bundle. Then the restriction $\sigma\vert_A\colon A\rightarrow Y$ is surjective. Since $A$ is a rational homogeneous space of Picard number one, according to \cite[Main Theorem]{HwangMok1999}, then either $Y\simeq \bbP^{n-1}$ or $\sigma\vert_A\colon A\rightarrow Y$ is an isomorphism. Moreover, if $Y$ is isomorphic to $\bbP^{n-1}$, by Lemma \ref{Section-Criterion}, $A$ is also a section of $\sigma$. In particular, $\sigma\colon X\rightarrow Y$ is actually a $\bbP^1$-bundle. Then \cite[Lemma 3.9]{CasagrandeDruel2015} shows that $X$ is isomorphic to $\bbP(Y,\sO_{Y}\oplus\sL^*)$, where $\sL$ is an ample line bundle over $Y$ and $A$ identifies with the section of $\sigma$ corresponding to $\sO_Y\oplus \sL^*\twoheadrightarrow \sL^*$. It follows that we have $\sL^*\simeq \sN_{A/X}\simeq \sO_{A}(-d)$ and we are in case (1.1).
    
    We assume now that $\sigma$ is a divisorial contraction. Thanks to Lemma \ref{Smoothness-Criterion}, the restriction map $\sigma\vert_A\colon A\rightarrow A_Y\defeq \sigma(A)$ is an isomorphism. On the other hand, as $\rho(Y)=1$, $A_Y$ is an ample divisor in $Y$. Then the pair $(Y,A_Y)$ is one of the possibilities listed in Theorem \ref{Watanabe-Homogeneous}. Moreover, as $A\cdot R_1>0$, $C$ is a smooth hypersurface in $A_Y$. In particular, there exists a positive integer $d'$ such that $C\in \vert\sO_{A_Y}(d')\vert$. Then a straightforward computation shows that $\sN_{A/X}$ is isomorphic to $\sN_{A_Y/Y}\otimes \sO_{A_Y}(-d')$. Thus, we get $d'=d+s$, where $\sN_{A_Y/Y}\simeq \sO_{A_Y}(s)$ and we are in case (1.2).
\end{proof}

\subsection{Case $\rho(X)=3$}

Let $X$ be a normal projective variety. We denote by $N_1(X)$ the vector space of $1$-cycles, with real coefficients, modulo numerical equivalence. For any closed subset $Z\subset X$, we denote by $N_1(Z,X)$ the subspaces of $N_1(X)$ generated by classes of curves contained in $Z$. The following result due to Casagrande and Druel provides a classification of Fano manifolds $X$ of maximal Picard number containing a prime divisor $A$ with $\dim N_1(A,X)=1$; see \cite{Fujita2012} for related results.

\begin{thm}\cite[Lemma 3.1 and Theorem 3.8]{CasagrandeDruel2015}
		\label{C-D-Classification}
	Let $X$ be a Fano manifold of dimension $n\geq 3$ and let $A\subset X$ be a prime divisor with $\dim N_1(A,X)=1$. Then $\rho(X)\leq 3$. Moreover, if $\rho(X)=3$, then $X$ is isomorphic to the blow-up of a Fano manifold $Y\simeq \bbP(\sO_Z\oplus\sO_Z(a))$ along an irreducible submanifold of dimension $(n-2)$ contained in a section of the $\bbP^1$-bundle $\pi\colon Y\rightarrow Z$, where $Z$ is a Fano manifold of dimension $(n-1)$ and $\rho(Z)=1$.
\end{thm}

Though we are interested in the case where $A$ is a negative divisor, to prove the second part of Theorem \ref{Classification-Negative-Divisors}, we still need to deal with the case where $A$ is a nef divisor. In particular, we prove the following preliminary result, which may be of independent interest.

\begin{prop}\label{Nef-Divisor}
	Let $X$ be an $n$-dimensional Fano manifold with $\rho(X)=2$ and $n\geq 3$. Assume that $X$ contains a nef divisor $A$ isomorphic to a rational homogeneous space of Picard number one. Let $\sO_A(1)$ be the ample generator of $\pic(A)$. Denote by $R_1$ and $R_2$ the extremal rays of $\NE{X}$ so that $A\cdot R_1>0$, and let $\sigma$ and $\pi$ be the associated extremal contractions, respectively. Assume moreover that $\sigma$ is a $\bbP^1$-bundle and $\pi$ is not small. Then one of the following statements holds.
	\begin{enumerate}
		\item $X$ is isomorphic to $\bbP(\sO_A\oplus\sO_{A}(-d))$ $(0\leq d<r)$, where $r$ is the index of $A$, and $A$ is a section with normal bundle $\sN_{A/X}\simeq \sO_A(d)$.
		
		\item $X$ is isomorphic to the blow-up of $\bbP^{n}$ at a point $x$ (or, equivalently, $X$ is isomorphic to the $\bbP^1$-bundle $\bbP(\sO_{\bbP^{n-1}}\oplus\sO_{\bbP^{n-1}}(-1))$, and $A$ is the strict transform of a smooth quadric hypersurface in $\bbP^{n}$ not containing $x$.
	\end{enumerate}
\end{prop} 

\begin{proof}
	By our assumption and \cite[Lemma 3.1]{CasagrandeDruel2015}, we have a diagram:
	\[\xymatrix{
		&   X\ar[dl]_{\sigma}\ar[dr]^{\pi}  &   \\
		Y   &                                   & Z
	}\] 
	where $Y$ is a Fano manifold and $\sigma$ is finite over $A$. Since $\pi$ is not small, by \cite[Proposition 3.3]{CasagrandeDruel2015}, either $\pi$ is a fiber type contraction onto $Z\simeq \bbP^1$, having $A$ as a fiber, or $\pi$ is a divisorial contraction sending its exceptional divisor $E$ to a point and $E\cap A=\emptyset$. If $\pi$ is a fiber type, then $X$ is isomorphic to $A\times \bbP^1$ (see \cite[Lemma 4.9]{Casagrande2009}) and we are in case (1) with $d=0$. 
	
	Now we shall assume that $\pi$ is birational. Then $Z$ is a Fano variety with only $\bbQ$-factorial terminal singularities so that $\rho(Z)=1$. Since $\pi$ is a birational map sending $E$ to a point, by \cite[Lemma 3.9]{CasagrandeDruel2015}, there exists an ample line bundle $\sO_Y(d)$ over $Y$, where $\sO_{Y}(1)$ is the ample generator of $\pic(Y)$ and $d>0$, such that $X$ is isomorphic to $\bbP(\sO_{Y}\oplus\sO_Y(-d))$ so that the exceptional divisor $E$ of $\pi$ identifies with the section corresponding to the projection $\sO_Y\oplus\sO_Y(-d)\twoheadrightarrow\sO_Y(-d)$. On the other hand, since $A$ is a rational homogeneous space of Picard number one and the restriction $\sigma\vert_A\colon A\rightarrow Y$ is surjective, according to \cite[Main Theorem]{HwangMok1999}, then either $A$ is a section of $\sigma$ or $Y$ is isomorphic to the projective space $\bbP^{n-1}$. If $A$ is a section of $\sigma$, then $Y$ is isomorphic to $A$. In particular, $X$ is isomorphic to $\bbP(\sO_A\oplus \sO_A(-d))$. On the other hand, since $A$ is disjoint from the negative section $E$, it follows that $A$ corresponds to a quotient $\sO_A\oplus \sO_{A}(-d)\twoheadrightarrow \sO_A$. In particular, we are in case (1) with $d>0$.
	
	In the sequel we shall assume that $Y$ is isomorphic to $\bbP^{n-1}$ and $A$ is not a section of $\sigma$. Then $X$ is isomorphic to $\bbP(\sO_{\bbP^{n-1}}\oplus\sO_{\bbP^{n-1}}(-d))$. Denote by $L$ the pull-back $\pi^*\cO_{\bbP^{n-1}}(1)$. Then there exist $x,y\in \bbQ$ such that $A\sim_{\bbQ} xL+yE$ because $X$ is a Fano manifold with $\rho(X)=2$. On the other hand, note that we have
	\begin{center}
		$0=A\cdot E\cdot L^{n-2}=x-yd$\qquad and\qquad $y=A\cdot L^{n-1}\in\bbZ_{>0}$.
	\end{center}
	Thus, both $x$ and $y$ are positive integers and $x=yd$. Set $H=dL+E$. Then $A\sim yH$. Since $A$ is not a section of $\sigma$, we must have $y\geq 2$. On the other hand, as $H\vert_E\equiv 0$ and $\pi$ is an extremal contraction, by the Cone Theorem (see \cite[Theorem 3.7]{KollarMori1998}), there exists a line bundle $H_Z$ on $Z$ such that $H=\pi^*H_Z$. Then we have 
	\[A_Z\defeq \pi_*A\sim yH_Z.\]	
	Since $A$ is disjoint from $E$ and $\pi$ is an isomorphism outside $E$, $A_Z$ is contained in the smooth locus of $A$ and it is isomorphic to $A$. As $\rho(A)=1$ and $Z$ is $\bbQ$-factorial, $A_Z$ is an ample Cartier divisor on $Z$. Moreover, since $Z$ is a Fano variety, by Kawamata-Viehweg vanishing theorem, we have $h^1(Z,\sO_Z)=0$. Then, by Proposition \ref{Ampleness-Criterion}, the line bundle $\sO_{Z}(A_Z)$ is very ample. Denote by $\Phi\colon Z\rightarrow \bbP^N$ the embedding defined by $\vert \sO_{Z}(A_Z)\vert$. 
	
	As $y\geq 2$, we have $A_Z\cdot C= yH_Z\cdot C\geq 2$ for any curve $C\subset Z$. In particular, $\Phi(Z)$ is not a cone over $\Phi(A_Z)$. Moreover, note that the pair $(Z,\sO_{Z}(A_Z))$ is not isomorphic to $(Q^n,\sO_{Q^n}(1))$ as $y\geq 2$. By Proposition \ref{Quadric-Extension}, we have 
	\[h^1(A_Z,T_{A_Z}\otimes \sO_{A_Z}(-A_Z))\not=0.\]
	Therefore, by \cite[Theorem B]{MerkulovSchwachhoefer1999}, the possibilities of the pair $(A_Z,\sO_{A_Z}(A_Z))$ are as follows:
	\[(\bbP^2,\sO_{\bbP^2}(3)), (Q^{n-1},\sO_{Q^{n-1}}(2)) (n\geq 4).\]
	As $y\geq 2$ and $H_Z$ is Cartier, one can easily see that $\sO_{A_Z}(H_Z)$ is the ample generator of $\pic(A_Z)$. On the other hand, note that the case $(\bbP^2,\sO_{\bbP^2}(3))$ can not happen, because in this case we have
	\[1=\sO_{\bbP^2}(1)^2=H_Z^2\cdot A_Z=3H_Z^3,\]
	which is impossible. If $(A_Z,\sO_{A_Z}(A_Z))$ is isomorphic to $(Q^{n-1},\sO_{Q^{n-1}}(2))$, then we have 
	\[H_Z^n=\frac{1}{2} A_Z\cdot H_Z^{n-1}=\frac{1}{2}\sO_{Q^{n-1}}(1)^{n-1}=1.\]
	Furthermore, since $Z$ is a Fano manifold with $\dim(Z)=n\geq 4$ and $H_Z$ is ample, by Kawamata-Viehweg vanishing theorem, we have $h^1(Z,\sO_{Z}(H_Z))=0$. As a consequence, from the following exact sequence
	\[0\rightarrow \sO_Z(-H_Z)\rightarrow \sO_Z(H_Z)\rightarrow \sO_{A_Z}(H_Z)\rightarrow 0,\] 
	we obtain
	\[h^0(Z,\sO_Z(H_Z))=h^0(A_Z,\sO_{A_Z}(H_Z))=h^0(Q^{n-1},\sO_{Q^{n-1}}(1))=n+1.\]
	Then, according to \cite[Chapter I, Theorem 1.1]{Fujita1990}, the pair $(Z,\sO_Z(H_Z))$ is isomorphic to $(\bbP^n,\sO_{\bbP^n}(1))$. We claim that $d=1$ in this case. Indeed, note that $\sO_X(E)$ is the tautological bundle $\sO_{\bbP(\sE)}(1)$, where $\sE\simeq \sO_{\bbP^{n-1}}\oplus \sO_{\bbP^{n-1}}(-d)$, so we have 
	\begin{equation}\label{KX1}
		K_X\sim -(n+d)L-2E.
	\end{equation}
	On the other hand, note that $\sN_{E/X}\simeq \sO_{\bbP^{n-1}}(-d)$, by the adjunction formula, we have
	\begin{equation}\label{KX2}
		K_X\sim_{\bbQ} -(n+1)\pi^*H_Z+\frac{n-d}{d} E=-(n+1)(dL+E)+\frac{n-d}{d}E.
	\end{equation}
	Combining \eqref{KX1} and \eqref{KX2} yields
	\[(-n-d+d(n+1))L\sim_{\bbQ} \left(-(n-1)+\frac{n-d}{d}\right)E.\]
	This is possible if and only if $d=1$ because $L$ is not numerically proportional to $E$. Hence, $X$ is isomorphic to $\bbP(\sO_{\bbP^{n-1}}\oplus\sO_{\bbP^{n-1}}(-1))$ and we are in case (2).
\end{proof}

Now we are ready to prove the second part of Theorem \ref{Classification-Negative-Divisors}. It can be regarded as a refinement of Theorem \ref{C-D-Classification}.

\begin{proof}[Proof of Theorem \ref{Classification-Negative-Divisors} (2)]
	By the proof of \cite[Theorem 3.8]{CasagrandeDruel2015}, there exists a blow-up $\sigma\colon X\rightarrow Y$ along a smooth center $C$ of codimension $2$, $Y$ is smooth and Fano, and $A\cdot R>0$, where $R$ is the extremal ray of $\NE{X}$ generated by the class of a curve contracted by $\sigma$. Moreover, there exists a Fano manifold $Z$ of dimension $n-1$, $\rho(Z)=1$ and a $\bbP^1$-bundle $\pi\colon Y\rightarrow Z$. Set $A_Y\defeq \sigma(A)$. Thanks to Lemma \ref{Smoothness-Criterion}, the restriction $\sigma\vert_{A}\colon A\rightarrow A_Y$ is an isomorphism. Note that $C$ is contained in $A_Y$ and we will denote by $d'$ the positive integer such that $C\in \vert \sO_{A_Y}(d')\vert$.
	
	First suppose that $A_Y$ is not nef in $Y$. Then the pair $(Y,A_Y)$ is isomorphic to one of the varieties listed in Theorem \ref{Classification-Negative-Divisors} (1.1); that is, $Y\simeq \bbP(\sO_Z\oplus\sO_Z(s))$ $(-r<s<0)$ and $A_Y$ is a section of $\pi$ with normal bundle $\sN_{A_Y/Y}\simeq \sO_Z(s)$. Moreover, the normal bundle $\sN_{A/X}\simeq\sO_{A}(-d)$ is isomorphic to $\sO_{Z}(s-d')$. Thus it follows that $d'=s+d$ and $s>-d$. Thus, we are done in this case.
	
	We assume now that $A_Y$ is nef in $Y$. By \cite[Proposition 3.3]{CasagrandeDruel2015}, $Y$ does not admit small contractions. Therefore the pair $(Y,A_Y)$ is isomorphic to one of the varieties listed in Proposition \ref{Nef-Divisor}. 
	
	We claim that case (2) of Proposition \ref{Nef-Divisor} cannot happen. Otherwise, $Y$ is isomorphic to the blow-up of $\bbP^n$ at a point $x$. Denote by $\mu\colon Y\rightarrow \bbP^n$ the blow-up. Then $C$ is contained in a section $G_Y$ of $\pi\colon Y\rightarrow Z$. We note that $G_Y$ is a the strict transform of a hyperplane $H$ passing through $x$ under $\mu$. In particular, $C$ is contained in $A_Y\cap G_Y$ and consequently $C$ is a hyperplane section of $A_Y\simeq Q^{n-1}$. Then a straightforward computation shows that $\sN_{A/X}$ is isomorphic to $\sO_{Q^{n-1}}(1)$, a contradiction.
	
	Finally suppose that we are in case (1) of Proposition \ref{Nef-Divisor}; that is, the pair $(Y,\sO_Y(A_Y))$ is isomorphic to $\bbP(\sO_Z\oplus\sO_Z(s))$ $(0\leq s<r)$ and $A_Y$ is a section with normal bundle $\sN_{A_Y/Y}\simeq \sO_{Z}(s)$. Then the normal bundle $\sN_{A/X}$ is isomorphic to $\sO_Z(s-d')$. Thus we have $d'=s+d$ and $-d<s$.
\end{proof}

\appendix

\section{Corrigendum}

\label{corrigendum}

The purpose of this note is to make a correction to \cite{Liu2020a}. In \cite[Theorem 1.2]{Liu2020a}, we give a classification of pairs $(X,A)$ such that $X$ is a Fano manifold of dimension $n\geq 3$ and $A$ is a smooth ample divisor which is isomorphic to some rational homogeneous space of Picard number $1$ and the conormal bundle $\sN_{A/X}^*$ is ample. However, it turns out that there exists one case missed in the statement of the theorem and \cite[Theorem 1.2]{Liu2020a} should be read as follows.

\begin{thm}\label{Classification-Negative-Divisors-appendix}
	Let $X$ be a Fano manifold of dimension $n\geq 3$ containing a divisor $A$ isomorphic to a rational homogeneous space with Picard number one. Denote by $\sO_A(1)$ the ample generator of $\pic(A)$ and by $r$ the index of $A$. Assume that $\sN_{A/X}$ is isomorphic to $\sO_A(-d)$ for some integer $d>0$. Then $0<d<r$ and we are in one of the following cases.
	\begin{enumerate}
		\item  $\rho(X)=2$ and the pair $(X,A)$ is isomorphic to one of the following:
		\begin{enumerate}
			\item $X$ is isomorphic to $\bbP(\sO_A\oplus \sO_A(-d))$ and $A$ is a section with normal bundle $\sN_{A/X}\simeq \sO_A(-d)$;
			
			\item $X$ is obtained by blowing up one of the pairs $(X',A')$ listed \cite[Theorem 1]{Watanabe2008} along a smooth center $C\in \vert \sO_{A'}(d+s)\vert$, where $\sO_{A'}(1)$ is the ample generator of $\pic(A')$, $\sN_{A'/X'}\simeq \sO_{A'}(s)$ and $A$ is the strict transform of $A'$.
			
			{\color{red}\item $X$ is a smooth element in $|\sO_{\bbP^{n-1}}(\sE)(2)\otimes \pi^*\sO_{\bbP^{n-1}}(2)|$ and $A$ is isomorphic to a quadric hypersurface such that $X\cap F=A$, where $\sE$ is the vector bundle $\sO_{\bbP^{n-1}}(d)\oplus \sO_{\bbP^{n-1}}\oplus \sO_{\bbP^{n-1}}(-1)$, the map $\pi:\bbP(E)\rightarrow \bbP^{n-1}$ is the natural projection and the variety $F\subset \bbP(E)$ is the subbundle corresponding to the quotient $\sE\rightarrow \sO_{\bbP^{n-1}}\oplus \sO_{\bbP^{n-1}}(-1)$.}
		\end{enumerate}
		
		\item $\rho(X)=3$ and $X$ is obtained by blowing up a Fano manifold $Y$ along a smooth center $C\in \vert \sO_{A_Y}(d+s)\vert$ such that $-d<s<r$, where $Y$ is isomorphic to $\bbP(\sO_A\oplus\sO_A(s))$, $A_Y$ is a section with normal bundle $\sN_{A_Y/Y}\simeq \sO_{A}(s)$, $\sO_{A_Y}(1)$ is the ample generator of $\pic(A_Y)$ and $A$ is the strict transform of $A_Y$.
	\end{enumerate}	
\end{thm}

The mistake appears in the proof of \cite[Lemma 3.2]{Liu2020a} and the statement of \cite[Lemma 3.2]{Liu2020a} is false in general. Indeed, in the proof of \cite[Lemma 3.2]{Liu2020a}, the value of $x$ should be
\[
x=\frac{ye}{2+\alpha e} = \frac{e}{2d+(r-d)e},
\]
while in the published paper $"e"$ in the denominator disappeared. In particular, the last equation in the same page should be as $2=2$ which is trivial. We correct \cite[Lemma 3.2]{Liu2020a} in Lemma \ref{Section-Criterion-appendix} by proving a weaker statement; that is, the number $2d/e$ is an integer. In particular, for $A$ being a rational homogeneous space of Picard number $1$, Lemma \ref{Section-Criterion-appendix} can be applied to show that $A$ is actually a section of the conic bundle $f:X\rightarrow \bbP^{n-1}$ unless it is isomorphic to a quadric hypersurface or the $10$-dimensional spinor variety $\bbS_{5}$. Then by a detailed analysis of the conic bundle structure $f$, we exclude the spinor variety $\bbS_{5}$ case by an ad-hoc argument.

Here is the organisation of this short note. In Section \ref{s.examples} we give an explicit construction of examples for the new case (1.3) of Theorem \ref{Classification-Negative-Divisors-appendix}. In Section \ref{s.correctionoflemma3.2} we prove a weaker statement of \cite[Lemma 3.2]{Liu2020a} to show that $2d/e$ is an integer and then applying it to show that in \cite[Lemma 3.2]{Liu2020a} if $A$ is assumed to be a rational homogeneous space of Picard number $1$, then $A$ is a section of $f$ unless $A$ is isomorphic to a quadric hypersurface. Finally we finish the proof of Theorem \ref{Classification-Negative-Divisors-appendix} by pointing out the parts affected by \cite[Lemma 3.2]{Liu2020a}.

\subsection{Examples}

\label{s.examples}

In this subsection, we construct some examples for case (1.3) of Theorem \ref{Classification-Negative-Divisors-appendix}. We start from the following example (see \cite[Proposition 3.4 (2)]{Liu2020a}). Let $\sF\rightarrow \bbP^n$ be the vector bundle $\sO_{\bbP^n}\oplus \sO_{\bbP^n}(-1)$ with $n\geq 3$. Then $F=\bbP(\sF)$ is isomorphic to the blowing-up of $\bbP^{n+1}$ at a point. Denote by $\mu\colon F\rightarrow \bbP^n$ the blowing-up and let $D$ be the exceptional divisor. Denote by $\zeta_F$ the tautological divisor of $\bbP(\sF)$ and by $\pi_F:F\rightarrow \bbP^n$ the natural projection. Let $H_F$ be a Weil divisor associated to the pull-back $\pi^*_F\sO_{\bbP^n}(1)$. Then we have
\[
\mu^*\sO_{\bbP^n}(1)\cong \sO_{\bbP(\sF)}(\zeta_F+H_F) \quad {\rm and}\quad D=\zeta_F.
\]
Let $A\subset F$ be a general smooth member in $|2\zeta_F+2H_F|$ such that $A$ is disjoint from $D$. Then $A$ is isomorphic to an $n$-dimensional quadric hypersurface. Consider the the vector bundle $\sE\rightarrow \bbP^n$ which is isomorphic to $\sO_{\bbP^n}(d)\oplus \sF$ with $1\leq d\leq n-1$. Then $F\subset \bbP(E)$ is a smooth prime divisor. Denote by $\zeta$ the tautological divisor of $\pi:\bbP(\sE)\rightarrow \bbP^n$ and by $H$ a Weil divisor associated to the pull-back $\pi^*\sO_{\bbP^n}(1)$. Then we have
\[
F\in  |\zeta - dH|.
\]
Recall that the restriction $\zeta|_{F}$ is isomorphic to $\zeta_F$ and $H|_{F}=H_F$. Consider the following short exact sequence
\[
0\rightarrow \sO_{\bbP(\sE)}(\zeta+(d+2)H) \rightarrow \sO_{\bbP(\sE)}(2\zeta+2H) \rightarrow \sO_{F}(2\zeta_F+2H_F)\rightarrow 0,
\]
As $K_{\bbP(\sE)}=-3\zeta+(d-2-n)H$, we have 
\[
\zeta + (d+2)H = K_{\bbP(\sE)} + 4\zeta + (n+4)H.
\]
As $d\geq 1$ and $n\geq 3$, $4\zeta+(n+4)H$ is ample. By Kodaira's vanishing theorem, we have $H^1(\bbP(\sE),\sO_{\bbP(\sE)}(\zeta+(d+2)H)=0$. In particular, the induced morphism 
\[
H^0(\bbP(\sE),\sO_{\bbP(\sE)}(2\zeta+2H)) \rightarrow H^0(F,\sO_{F}(2\zeta_F+2H_F))
\]
is surjective and there exists a divisor $X\in |2\zeta+2H|$ such that $X\cap F=A$. Moreover, as $A$ is general and $2\zeta+2H$ is globally generated, we may assume that $X$ is again smooth. Note that we have
\[
\sO_X(A)=\sO_{\bbP(\sE)}(F)|_X\cong \sO_{X}(\zeta - dH).
\]
On the other hand, as $\zeta|_A=\zeta_F|_A\sim D|_A=0$, we get
\[
\sO_X(A)|_A \cong \sO_A(-d).
\]
Now we claim that $X$ is a Fano manifold. By adjunction formula, we have
\[
K_X=(K_{\bbP(\sE)}+ 2\zeta+2H)|_X=\sO_X(-\zeta+(d-n)H).
\]
If $d\leq n-1$, then $\zeta+(n-d)H$ is a semi-ample big and nef divisor with non-ample locus contained in $D$. By our construction, the variety $X$ is disjoint from $D$, thus the restriction $(\zeta+(n-d)H)|_X$ is ample and hence $-K_X$ is ample.

\subsection{Correction of \cite[Lemma 3.2]{Liu2020a}}
\label{s.correctionoflemma3.2}

As pointed out in the beginning, \cite[Lemma 3.2]{Liu2020a} is not correct in general. We replace it by the following weaker statement.

\begin{lemma}
	\label{Section-Criterion-appendix}
	Let $X$ be a Fano manifold of dimension $n\geq 3$ and with $\rho(X)=2$, and let $A$ be a smooth Fano hypersurface of $X$ such that $\pic(A)\simeq \bbZ\sO_A(1)$ for some ample line bundle $\sO_A(1)$ and $\sN_{A/X}\simeq\sO_A(-d)$ for some $d>0$. Assume furthermore that there exists a curve of degree $1$ on $A$; i.e. an irreducible curve $C\subset A$ such that $c_1(\sO_A(1))\cdot C=1$. If $X$ admits an extremal contraction $f\colon X\rightarrow \bbP^{n-1}$, which is a conic bundle, such that $f$ is finite over $A$. {\color{red} Then $f^*\sO_{\bbP^{n-1}}(1)\cong \sO_A(1)$ and $2d/e$ is an integer, where $e=\deg_A(\sO_A(1))$.} 
\end{lemma}
\begin{proof}
	Denote by $r$ the index of $A$, i.e., $\sO_A(-K_A)\simeq \sO_A(r)$. As $X$ is Fano, the line bundle $\sO_{A}(-K_X)\simeq \sO_{A}(r-d)$ is ample. We get $r>d$. Since $A$ is not nef and $X$ is Fano, there exists an extremal ray $R$ of $\NE{X}$ such that $A\cdot R<0$. Let $\pi\colon X\rightarrow Y$ be the associated contraction. Then $\Exc(\pi)\subset A$ as $A\cdot R<0$. On the other hand, every curve contained in $A$ has class in $R$ since $\rho(A)=1$ and $\sO_X(A)\vert_A\simeq \sN_{A/X}$ is negative. This implies that $A=\Exc(\pi)$ and that $\pi(A)$ is a point. By adjunction, we have
	\[K_{X}\sim_{\bbQ}\pi^*K_Y+\frac{r-d}{d} A.\]
	Let $H$ be a Weil divisor associated to $f^*\sO_{\bbP^{n-1}}(1)$. Since $\rho(X)=2$ and the contraction map $f$ is supposed to be elementary, there exist $x,y\in \bbQ$ such that 
	\[H\equiv x\pi^*(-K_Y)-yA.\]
	Denote by $e$ the degree $\sO_A(1)^{n-1}$. Set $\alpha=(r-d)/d$ and $m\defeq(-K_Y)^n$. Then we get
	\begin{equation}\label{Eq 1-appendix}
		0=H^n=x^n m - y^n d^{n-1} e
	\end{equation}
	and
	\begin{equation}\label{Eq 2-appendix}
		2=(-K_X)\cdot H^{n-1}=x^{n-1}m - \alpha y^{n-1}d^{n-1}e.
	\end{equation}
	Set $l\defeq 2d^2/(yd)^{n-1}$. Now we follow the argument of \cite[Lemma 1]{Tsukioka2006} to show that $yd=1$. By \eqref{Eq 1-appendix}, we have
	\[
	\left(\frac{y}{x}\right)^n=\frac{m}{d^{n-1}e}.
	\]
	Combining \eqref{Eq 1-appendix} and \eqref{Eq 2-appendix} yields
	\[
	x=\frac{y^n d^{n-1} e}{2+\alpha y^{n-1} d^{n-1} e}.
	\]
	This implies
	\[
	\frac{y}{x}=y\cdot \frac{2 + \alpha y^{n-1} d^{n-1} e}{y^n d^{n-1} e} = \frac{2 + \alpha y^{n-1} d^{n-1} e}{y^{n-1} d^{n-1} e}.
	\]
	It follows
	\[
	\left(\frac{2d^2}{y^{n-1} d^{n-1}} + \alpha d^2 e\right)^n = \frac{m}{d^{n-1} e} \cdot (d^2 e)^n = md^n e^{n-1} \cdot d.
	\]
	As $md^n$ is an integer, it follows that
	\[
	\frac{2d^2}{y^{n-1} d^{n-1}} + \alpha d^2 e = \frac{2d^2}{y^{n-1} d^{n-1}} + (r-d)de
	\]
	is an integer. In particular, $l$ is an integer. As $d<n$, we obtain
	\[
	2(n-1)^2 \geq 2d^2 = (yd)^{n-1} \cdot l.
	\]
	As $H\cdot C = -y A\cdot C=-yc_1(\sO_A(A))\cdot C= yd$ is an integer and $n\geq 3$, we must have $yd\leq 2$. Moreover, if $yd=2$, we have $2d^2=2^{n-1}\cdot l$, hence $d^2=2^{n-2}\cdot l$. On the other hand, as $(l+(r-d)de)^n=md^n e^{n-1} d\in d\mathbb{N}$, we have $l^n\in d\bbN$. In particular, as $d\leq n-1$, we obtain $(n,d,l)\in \{(3,2,2),(5,4,2)\}$. If $(n,d,l)\in \{(3,2,2),(5,4,2)\}$, then we must have $r>d=n-1=\dim(A)$. By Kobayashi-Ochiai's theorem, then $A$ is isomorphic to $\bbP^{n-1}$, which is impossible by \cite[Lemma 1]{Tsukioka2006}. Thus, we have $yd=1$ and as a consequence, we have
	\[A\cdot H^{n-1}=A\cdot(x\pi^*(-K_Y)-yA)^{n-1}=(yd)^{n-1}e=e.\]
	In particular, we get $H|_A\cong \sO_A(1)$. {\color{red}As a consequence, we obtain
		\[
		x=\frac{ye}{2+\alpha e}=\frac{e}{2d+(r-d)e}\quad {\rm and}\quad \frac{y}{x}=\frac{2+\alpha e}{e}=\frac{2d+(r-d)e}{ed}.
		\]
		It yields
		\[
		K_X=-\frac{1}{x}H-\frac{y}{x}A+\frac{r-d}{d}A=-\frac{2d+(r-d)e}{e}H - \frac{2}{e}A.
		\]
		This implies
		\begin{align*}
			K_X^2\cdot H^{n-2} & = 2\cdot \left[\frac{2d+(r-d)e}{e}\right]\cdot \left(\frac{2}{e}\right)A\cdot H^{n-1} + \left(\frac{2}{e}\right)^2 A^2\cdot H^{n-2} \\
			& = 4 \left[\frac{2d+(r-d)e}{e}\right] - \frac{4d}{e}\\
			& = \frac{4d}{e} + 4(r-d).
		\end{align*}
		As $K_X^2\cdot H^{n-2}$ is an integer, it follows that $4d/e$ is an integer. On the other hand, set $\beta=\left[2d+(r-d)e\right]/e$, we also have
		\begin{align*}
			(-K_X)^3\cdot H^{n-3} & = 3\beta^2\cdot \frac{2}{e} AH^{n-1} + 3\beta\cdot \left(\frac{2}{e}\right)^2 A^2\cdot H^{n-2} + \left(\frac{2}{e}\right)^3 A^3\cdot H^{n-3}\\
			& = 6\beta^2 - \frac{12\beta d}{e} +\frac{8d^2}{e^2} \\
			& = 6 \left[\frac{2d+(r-d)e}{e}\right]^2 - \frac{12 d}{e} \cdot \left[\frac{2d+(r-d)e}{e}\right] + \frac{8 d^2}{e^2}   \\
			& = \frac{24 d^2}{e^2} + \frac{24d(r-d)}{e} + 6(r-d)^2 - \frac{24 d^2}{e^2} - \frac{12 d(r-d)}{e} + \frac{8 d^2}{e^2}\\
			& = \frac{12 d(r-d)}{e} + 6(r-d)^2 + \frac{8 d^2}{e^2}.
		\end{align*}
		As $(-K_X)^3\cdot H^{n-2}$ and $4d/e$ are integers, it follows that $8d^2/e^2$ is an integer. This implies that $2d/e$ is an integer. In particular, we have $e\leq 2d\leq 2r-2$.}
\end{proof}

The remaining part of this section is devoted to prove the following result, which will be used to finish the proof of Theorem \ref{Classification-Negative-Divisors-appendix}.

\begin{prop}
	\label{p.rhs-conic}
	In Lemma \ref{Section-Criterion-appendix}, if we assume in addition that $A$ is isomorphic to a rational homogeneous space of Picard number $1$, then $A$ is a section of $f$ unless $A$ is isomorphic to a quadric hypersurface.
\end{prop}

The proof of Proposition \ref{p.rhs-conic} above will be divided into two different parts. In the first part, we show that a rational homogeneous space $A$ of Picard number $1$ satisfies $e\leq 2r-2$ if and only if it is isomorphic to one of the following: a projective space, a quadric hypersurface, the Grassmann variety $\Gr(2,5)$ and the $10$-dimensional spinor variety $\bbS_{5}$. The projective space cases are proved in \cite{Tsukioka2006} and the Grassmann variety $\Gr(2,5)$ can be easily excluded by the fact that $2d/e$ is an integer for some $d\leq r-1$. In the second part, we exclude the spinor variety $\bbS_{5}$ case by studying the conic bundle structure $f$ carefully.

\subsubsection{Rational homogeneous space of small degrees}

Now we proceed to classify rational homogeneous spaces of Picard number $1$ satisfying $e\leq 2r-2$. 

\begin{prop}
	\label{Prop:eleq2d-2}
	Let $A$ be an $n$-dimensional rational homogeneous space of Picard number $1$ with degree $e$ and index $r$. Then $e\leq 2r-2$ if and only if $A$ is isomorphic to one of the following varieties:
	\begin{enumerate}
		\item a projective space $\bbP^n$ with $e=1$ and $r=n+1$;
		
		\item a quadric hypersurface $\bbQ^n$ $(n\geq 3)$ with $e=2$ and $r=n$;
		
		\item the Grassmann variety $\Gr(2,5)$ with $e=5$ and $r=5$;
		
		\item the $10$-dimensional spinor variety $\bbS_{5}$ with $e=12$ and $r=8$.
	\end{enumerate}
\end{prop}

\begin{thm}
	\cite{Ionescu2008}
	\label{Thm:eleqN}
	Let $Z\subsetneq \bbP^N$ be an $n$-dimensional irreducible, smooth, non-degenerate and linearly normal projective variety of degree $e$. Assume that $Z$ is a Fano manifold of Picard number $1$ such that $2N\geq 3n$ and $n\geq 2$. If $e\leq N$, then $Z$ has index at least $n-2$.
\end{thm}

\begin{proof}
	Denote by $c=N-n$ the codimension of $Z$. Then we have $n\leq 2c$ by assumption. Firstly we assume that $n\leq c+1$. Then we have 
	\[
	e\leq N \leq N+c+1-n=2c+1.
	\]
	As $Z$ has Picard number $1$, it follows from \cite[Theorem I]{Ionescu1985} that $Z$ has index at least $n-1$. 
	
	Secondly we assume that $c+2\leq n\leq 2c$. Let $\Delta$ be the $\Delta$-genus $e-c-1$ of $Z$. If $\Delta\leq 1$,  it is well known from the classification of Fano manifolds that $X$ is has index $\geq n-1$ (see for instance \cite[Theorem A and B]{Ionescu2008}). Thus we may assume that $\Delta\geq 2$. Then it follows from \cite[Propoistion 10]{Ionescu2008} that $Z$ has index $n-2$.
\end{proof} 

\begin{lemma}
	\label{Lemma:rgeqn-2}
	Let $A$ be an $n$-dimensional rational homogeneous space of Picard number $1$. If $r\geq n-2$, then $A$ is isomorphic to one of the following
	\begin{center}
		$\bbP^n$,\, $Q^n$ $(n\geq 3)$,\, $\Gr(2,5)$,\, $\bbS_{5}$,\, $\Gr(2,6)$,\, $LG(3,6)$,\, $G_2/P_2$.
	\end{center}
	In particular, $e\leq 2r-2$ if and only $A$ is isomorphic to one of the varieties listed in Proposition \ref{Prop:eleq2d-2}.
\end{lemma}

\begin{proof}
	This is well-known from the classification of Fano manifold of index at least $ n-2$, see \cite[Theorem 3.1.14, Table 12.1 and Theorem 5.2.3]{IskovskikhProkhorov1999}. In particular, the corresponding pairs $(e,r)$ are as follows
	\begin{center}
		$(1,n+1)$,\, $(2,n)$,\, $(5,5)$,\, $(12,8)$,\, $(14,6)$,\, $(16,4)$\, $(18,3)$.
	\end{center} 
	This finishes the proof.
\end{proof}

\begin{lemma}
	\label{Lemma:r>h+1}
	Let $X=\mathcal{D}_l/P_k$ be a rational homogeneous space of Picard number $1$, with index $r$. Let $L$ be the ample generator of $\pic(X)$. Then $2r> h^0(X,L)+1$ if and only if $X$ is isomorphic to either $\bbP^n$ or $\bbQ^n$ $(n\geq 4)$.
\end{lemma}

\begin{proof}
	We refer the reader to \cite[Table 1]{Konno1986} for the explicit values of $r$ and $h^0(X,L)$ in terms of $l$ and $k$. We just remark that in \cite[Table 1]{Konno1986}, the index of $X$ is denoted by $k$ and the node is denote by $r$. Moreover, we also recall that $G_2/P_1$ is isomorphic to the $5$-dimensional quadric hypersurface $\bbQ^5$. In particular, if $X$ is of E-F-G type, it can be easily shown that $2r>h+1$ if and only if $X$ is isomorphic to $G_2/P_1$, where $h=h^0(X,L)$. Now we prove it for $X$ being of classical type. In the following table, we collect the values of $r$ and $h$ for $X$ of classical types. Here we remark that $B_l/P_l$ is isomorphic to $D_l/P_{l-1}$ and it is also isomorphic to $D_l/P_l$ which is called the spinor variety $\bbS_l$, and $C_l/P_1$ is isomorphic to $A_{2l-1}/P_1$ which is the projective space $\bbP^{2l-1}$.
	
	\renewcommand*{\arraystretch}{1.6}
	\begin{longtable}{|M{1cm}|M{3cm}|M{3cm}|M{3cm}|}
		\hline	
		
		$\mathcal{D}_l$            &    node $k$                &
		$r$                        &    $h$                
		\\
		\hline
		
		$A_l$                      &   $1\leq k\leq l$           &
		$l+1$                      &   $\binom{l+1}{k}$                
		\\
		\hline
		
		$B_l$                      &   $1\leq k\leq l-1$         &
		$2l-k$                     &   $\binom{2l+1}{k}$                
		\\
		\hline
		
		$C_l$                  &   $2\leq k\leq l$         &
		$2l+1-k$               &   $\binom{2l}{k}-\binom{2l}{k-2}$                
		\\
		\hline
		
		$D_l$                  &   $1\leq k\leq l-2$         &
		$2l-1-k$               &   $\binom{2l}{k}$                
		\\
		\hline
		
		$D_l$                  &   $l-1$                     &
		$2l-2$                 &   $2^{l-1}$                
		\\
		\hline
	\end{longtable}
	
	\begin{enumerate}
		\item $\mathcal{D}_l=A_l$. Firstly we note that $A_l/P_k$ is isomorphic to $A_l/P_{l-k+1}$. Thus we may assume that $2k\leq l+1$. Moreover, $A_l/P_1$ is isomorphic to the projective space $\text{Gr}(1,l+1)=\bbP^{l}$ and $A_3/P_2$ is isomorphic to $\Gr(2,4)$ which is the $4$-dimensional quadric hypersurface. For $2\leq k\leq \frac{l+1}{2}$, by our assumption, we have
		\[
		2r=2(l+1)>h=\binom{l+1}{k}\geq \binom{l+1}{2}=\frac{l(l+1)}{2}.
		\]
		This implies that $l=3$ and $k=2$; that is, $X$ is isomorphic to $\bbQ^4$.
		
		\item $\mathcal{D}_l=B_l$. Firstly we note that $B_l/P_1$ is isomorphic to the $(2l-1)$-dimensional quadric hypersurface $\bbQ^{2l-1}$. If $k\geq 2$, by our assumption, we have
		\[
		4l-4\geq 4l-2k=2r > h = \binom{2l+1}{k}\geq \binom{2l+1}{2}=l(2l+1)\geq 4l+2,
		\]
		which is obviously impossible.
		
		\item $\mathcal{D}_l=C_l$. Firstly we note that $C_2/P_2$ is isomorphic to the $3$-dimensional quadric hypersurface. By our assumption, we have $2r\geq h+2\geq \dim(X) + 3$ since $L$ is very ample. Recall that the dimension of $X$ is as follows:
		\[
		\dim(X)=2k(l-k)+\frac{k(k+1)}{2}.
		\]
		Thus, if $2\leq k\leq l-1$, then we have
		\begin{align*}
			2r=4l+2-2k & \geq 2k(l-k)+\frac{k(k+1)}{2} + 3\\
			& \geq (2k-4)l + 4l-2k^2+\frac{k(k+1)}{2} + 3\\
			& \geq (2k-4)(k+1) + 4l -2k^2 +\frac{k(k+1)}{2} +3\\
			& \geq 4l-1+\frac{k(k-3)}{2}
		\end{align*}
		which is possible only if $k=2$. Nevertheless, if $k=2$, then we have $2r=4l-2$ and $h+1=l(2l-1)$, which is impossible as $l\geq k+1=3$. Thus we may assume that $l=k$. Then we obtain
		\[
		2r=2l+2 > h+1\geq \dim(X)+2 = \frac{l(l+1)}{2}+2,
		\]
		which is impossible unless $l=2$. On the other hand, note that $C_2/P_2$ is isomorphic to the $3$-dimensional quadric hypersurface, which is again impossible.
		
		\item $\mathcal{D}_l=D_l$ and $1\leq k\leq l-2$. Firstly we note that $D_l/P_1$ is the $(2l-2)$-dimensional quadric hypersurface. For $k\geq 2$, by our assumption, we have
		\begin{align*}
			4l-6\geq 2r=2(2l-1-k) & > h=\binom{2l}{k}\geq \binom{2l}{2}=l(2l-1),
		\end{align*}
		which is impossible as $l\geq k+2\geq 4$.
		
		\item $\mathcal{D}_l=D_l$ and $k=l-1$. If $2\leq l\leq 4$, the variety $X$ is isomorphic to $\bbP^{1}$ $(l=2)$, $\bbP^3$ $(l=3)$ and the $6$-dimensional quadric hypersurface $\bbQ^6$ $(l=4)$. Thus, we may assume that $l\geq 5$. Then by our assumption, we obtain
		\[
		2r=2(2l-2)> h=2^{l-1}=4\cdot 2^{l-3}\geq 4\cdot 2(l-3),
		\]
		which is impossible.
	\end{enumerate}
	This finishes the proof.
\end{proof}

Now we are in the position to prove Proposition \ref{Prop:eleq2d-2}.

\begin{proof}[Proof of Proposition \ref{Prop:eleq2d-2}]
	Let $L$ be the ample generator of $\pic(A)$. Then $L$ is very ample. Denote $h^0(X,L)$ by $h$.
	
	If $2r>h+1$, by Lemma \ref{Lemma:r>h+1}, $A$ is isomorphic to either a projective space or a quadric hypersurface.
	
	If $2r\leq h+1$, then we get $e\leq 2r-2\leq h-1$ and therefore Theorem \ref{Thm:eleqN} implies that either $A$ has index $\geq n-2$, or $2(h-1)<3n$. In the former case, we can conclude by lemma \ref{Lemma:rgeqn-2}. In the latter case, we note that $A$ is quadric; that is, the embedding $A\subset \bbP(H^0(A,L))$ is scheme-theoretically cut out by quadric hypersurfaces. Then $A$ is actually a complete intersection in $\bbP(H^0(A,L))$ (cf. \cite{IonescuRusso2013}). Hence, $A$ is actually a quadric hypersurface.
\end{proof}

\subsubsection{Fano conic bundles}

Let $f: X\rightarrow \bbP^{n-1}$ be an $n$-dimensional Fano conic bundle with $n\geq 3$, i.e., $X$ is a Fano manifold and $f$ is a conic bundle structure. Denote by $\sE$ the locally free sheaf $f_*\sO_X(-K_X)$ of rank $3$. Let $\zeta$ be the tautological divisor of $\bbP(\sE)$. Denote by $c$ the integer such that $\det(\sE)\cong \sO_{\bbP^{n-1}}(c)$. Let $H$ be a Weil divisor associated to $\pi^*\sO_{\bbP^{n-1}}(1)$, where $\pi:\bbP(\sE)\rightarrow \bbP^{n-1}$ is the natural projection. Then $X$ can be embedded in $\bbP(\sE)$ as a divisor such that
\[
X\in |2\zeta + (n-c)H|.
\]
Let $A\subset X$ be an irreducible smooth divisor which is a Fano manifold of Picard number $1$ such that $H|_{A}$ is the ample generator of $\pic(A)$ and $\sO_X(A)|_A\cong \sO_A(-dH)$ for some $d>0$. Denote by $e$ the degree of $A$ with respect to $H|_A$ and by $h:A\rightarrow \bbP^{n-1}$ the induced finite morphism. Let $r$ be the index of $A$.

\begin{lemma}
	\label{Lemma:Spliting-type-I}
	Let $\sE\rightarrow \sO_{\bbP^{n-1}}(a)$ be a non-zero morphism of coherent sheaves. If $a\leq 0$, then there exists an integer $b\leq a$ such that $2b=c-n$ and
	\[
	\sE \cong \sO_{\bbP^{n-1}}(c-b-r+d) \oplus \sO_{\bbP^{n-1}}(r-d) \oplus \sO_{\bbP^{n-1}}(b).
	\]
\end{lemma}

\begin{proof}
	Let $\sQ\subset \sO_{\bbP^{n-1}}(a)$ be the image of $\sE$ and denote by $\sL\cong \sO_{\bbP^{n-1}}(b)$ the reflexive hull of $\sQ$. Then we have $b\leq a\leq 0$. In particular, the generically surjective morphism $\sE\rightarrow \sL$ defines a rational section $S\subset \bbP(\sE)$ such that there exists a Zariski open subset $U\subset \bbP^{n-1}$ satisfying
	\begin{enumerate}
		\item $\codim(\bbP^{n-1}\setminus U)\geq 2$;
		
		\item $S\cap \pi^{-1}(U)\rightarrow U$ is an isomorphism;
		
		\item $\sO_{\bbP(\sE)}(\zeta)|_{S\cap \pi^{-1}(U)} \cong \pi^*\sL|_{S\cap \pi^{-1}(U)}$.
	\end{enumerate}
	Take a log resolution $\mu:\widetilde{S}\rightarrow S$ such that $\mu$ is an isomorphism over $S\cap \pi^{-1}(U)$ and denote by $g:\widetilde{S}\rightarrow \bbP^{n-1}$ the induced birational morphism. Then we have
	\[
	\mu^*\sO_{\bbP(\sE)}(\zeta) \cong g^*\sO_{\bbP^{n-1}}(b)\otimes \sO_{\widetilde{S}}(\Delta),
	\]
	where $\Delta$ is a $g$-exceptional divisor. Since $\zeta$ is $\pi$-ample, the pull-back $\mu^*\zeta$ is $g$-nef. Then the negativity lemma implies that $-\Delta$ is effective.
	
	\bigskip
	\textbf{Claim 1.} \textit{Let $C\subset S$ be an irreducible projective curve such that $C\cap \pi^{-1}(U)\not=\emptyset$. Then we have $\zeta \cdot C\leq 0$.} 
	
	\bigskip
	
	\textit{Proof of Claim 1.} By assumption, the intersection $\mu(\Delta)\cap \pi^{-1}(U)$ is empty. Let $\widetilde{C}\subset \widetilde{S}$ be the strict transform of $C$. Then we have
	\[
	\zeta \cdot C = \mu^*\zeta \cdot \widetilde{C} = c_1(g^*\sO_{\bbP^{n-1}}(b))\cdot \widetilde{C} + \Delta\cdot \widetilde{C} \leq b\leq 0.
	\] 
	This finishes the proof of Claim 1.
	
	\bigskip
	
	Note that $\zeta|_X=-K_X$ is ample, thus Claim 1 implies that the image $\pi(X\cap S)$ is contained in $\bbP^{n-1}\setminus U$. In particular, let $l\subset U$ be a general line and let $\bar{l}\subset S$ be the section corresponding to the quotient $\sE|_l\rightarrow \sL|_l$. Then $X$ is disjoint from $\bar{l}$. In particular, we have 
	\[
	X\cdot \bar{l} = (2\zeta + (n-c)H)\cdot \bar{l}=2b+(n-c)=0.
	\]
	As a consequence, we have $2b=c-n$. 
	
	\bigskip
	
	\textbf{Claim 2.} \textit{The morphism $\sE\rightarrow \sL$ is surjective.} 
	
	\bigskip
	
	\textit{Proof of Claim 2.} Let $x\in X$ be an arbitrary point and let $x\in l$ be a general line passing through $x$ such that $l\cap U\not=\emptyset$. We consider the restriction \[
	\sigma_l: \sE|_l\rightarrow \sL|_l.
	\]
	We claim that $\sigma_l$ is surjective. Otherwise, let $\sQ_l$ be the image of $\sigma_l$. Then we must have $\sQ_l\cong \sO_{\bbP^1}(b')$ for some $b'<b$. Let $\bar{l}\subset \bbP(\sE|_l)$ be the section corresponding to the quotient $\sE|_l\rightarrow \sQ_l$. Then we obtain
	\[
	X\cdot \bar{l} = (2\zeta + (n-c)H)\cdot \bar{l} =2b'+n-c<2b+n-c=0.
	\]
	In particular, $\bar{l}$ is contained in $X$ and $\zeta \cdot \bar{l}=b'<0$, which is impossible as $\zeta|_X=-K_X$ is ample. This finishes the proof of claim 2.
	
	\bigskip
	
	\textbf{Claim 3.} \textit{The vector bundles $\sE$ splits as a direct sum of line bundles as follows
		\[
		\sO_{\bbP^{n-1}}(c-b-r+d) \oplus \sO_{\bbP^{n-1}}(r-d) \oplus \sO_{\bbP^{n-1}}(b).
		\]} 
	
	\textit{Proof of Claim 3.}
	Firstly note that we have $\zeta|_A\cong \sO_A(r-d)$. Thus $h^*\sE$ admits a quotient line bundle $h^*\sE\rightarrow \sO_A(r-d)$ with the corresponding section $A'\subset \bbP(h^*\sE)$ such that 
	\[
	\bar{h}(A')=A,
	\]
	where $\bar{h}:\bbP(h^*\sE)\rightarrow \bbP(\sE)$ is the induced morphism. 
	
	On the other hand, let $S'\subset \bbP(h^*\sE)$ be the section corresponding to the induced quotient line bundle $h^*\sE\rightarrow h^*\sL$. Then we have $\bar{h}(S')=S$. By Claim 2, $S$ is a section of $\bbP(\sE)\rightarrow \bbP^{n-1}$ such that $\zeta|_S\cong \sO_{\bbP^{n-1}}(b)$. This yields that $X$ is disjoint from $S$  and hence $A$ is disjoint from $S$. Thus, $A'$ is also disjoint from $S'$. Let $\sF\subset \sE$ be the kernel of $\sE\rightarrow \sL$.  Then the induced morphism $h^*\sF\rightarrow \sO_A(r-d)$ is surjective. As a consequence, we obtain the following exact sequence of vector bundles
	\[
	0\rightarrow \sG\rightarrow h^*\sF \rightarrow \sO_A(r-d) \rightarrow 0.
	\]
	As $A$ is a Fano manifold of Picard number $1$ and of dimension $\geq 2$, we must have $H^1(A,\sO_A(i))=0$ for any $i\in \bbZ$ by Kodaira's vanishing theorem. Then we obtain 
	\[
	h^*\sF\cong \sG\oplus \sO_A(r-d)\cong \sO_A(c-b-r+d)\oplus \sO_A(r-d).
	\]
	Then Lemma \ref{p.splitting} below implies that $\sF\cong \sO_{\bbP^{n-1}}(c-b-r+d)\oplus \sO_{\bbP^{n-1}}(r-d)$ and we are done.
\end{proof}

\begin{lemma}
	\label{p.splitting}
	Let $f:Y\rightarrow X$ be a finite morphism between Fano manifolds of Picard number $1$ with dimension at least $2$. Let $\sE$ be a vector bundle of rank $2$ over $X$. If $f^*\sE\cong \sL_1\oplus \sL_2$, then there exist line bundles $\sM_i$ on $X$ such that $f^*\sM_i\cong\sL_i$ for $i=1$, $2$ and $\sE\cong \sM_1\oplus \sM_2$.
\end{lemma}

\begin{proof}
	Firstly we assume that $f^*\sE$ is semistable. Then we have $\sL_1\cong \sL_2$ as $Y$ is Fano with $\rho(Y)=1$. In particular, the vector bundle $f^*\sE$ is numerically projectively flat (see \cite[Definition 4.1]{LiuOuYang2020}), so is $\sE$ itself. As $X$ is simply connected, it follows that $\sE$ is a direct sum $\sM_1\oplus \sM_2$ such that $\sM_1\cong \sM_2$. Then it is clear that we have $f^*\sM_1\cong  \sL_1$ as $\det(f^*\sE)=f^*\det(\sE)$.
	
	Next we assume that $f^*\sE$ is not semistable. Then $\sE$ itself is not semistable. Without loss of generality, we may assume that $c_1(\sL_1)>c_1(\sL_2)$. Let $\sM_1\subset \sE$ be the maximal destabilisor. Then we have $c_1(f^*\sM_1)>c_1(\sL_2)$. In particular, the induced morphism $f^*\sM_1\rightarrow f^*\sE$ factors through $\sL_1\rightarrow f^*\sE$; that is, $f^*\sM_1\subset \sL_1$. As $\sM_1$ is an invertible sheaf and $\sM_1$ is saturated in $\sE$, it follows that $f^*\sM_1\subset \sL_1$ is also saturated and hence $f^*\sM_1\rightarrow \sL_1$ is an isomorphism. Thus, the line bundle $\sM_1$ is a subbundle of $\sE$ and therefore $\sM_2:=\sE/\sM_1$ is a line bundle satisfying $f^*\sM_2\cong \sL_2$. In particular, as $X$ is a Fano manifold of Picard number $1$ with dimension at least $2$, it follows $H^1(X,\sM)=0$ for any line bundle $\sM$ over $X$, and hence $\sE\cong \sM_1\oplus \sM_2$.
\end{proof}

Now we assume that $A$ is the $10$-dimensional spinor variety. As $2d/e$ is an integer, $e=12$ and $d+1\leq r=8$, as computed in the proof of Lemma \ref{Section-Criterion-appendix}, we obtain
\begin{center}
	$e=12$, $d=6$, $r=8$, $\sO_A(\zeta)\cong \sO_A(2)$ and $\sO_X(A)\cong \sO_X(6\zeta - 18H)$. 
\end{center}
Moreover, we have the following equations:
\[
\begin{cases}
	\quad\quad K_X^2\cdot (H|_X)^9    & =\, \, \frac{4d}{e}+4(r-d)=10 \\
	(-K_X)^3\cdot (H|_X)^8 & =\, \, \frac{12d(r-d)}{e}+6(r-d)^2+\frac{8d^2}{e^2}=38\\
	(-K_X)^4\cdot (H|_X)^7 & =\, \, (-K_X)^3\cdot (3H|_X+\frac{1}{6}A)\cdot (H|_X)^{7}=130.
\end{cases}
\]
Denote by $L$ a general hyperplane section of $\bbP^{10}$.  We are ready to calculate the Chern classes of $\sE$. Recall that we have the following
\[
\zeta^3 = \pi^*c_1(\sE)\cdot \zeta^2 - \pi^*c_2(\sE)\cdot \zeta + \pi^*c_3(\sE).
\]
Firstly we have
\[
\zeta^2\cdot (2\zeta+(11-c)H)\cdot H^9= K_X^2\cdot (H|_X)^9=10. 
\]
This implies that $11+c=10$ and hence $c=\zeta^3\cdot H^9=-1$.

Secondly we have
\[
\zeta^3\cdot (2\zeta+ 12H)\cdot H^8 = (-K_X)^3\cdot (H|_X)^8 =38.
\]
This implies that $c_2(\sE)\cdot L^8=-24$. One can also calculate that $c_3(\sE)\cdot L^7=-36$, but we do not need it in the following so we leave it for the interested reader.

\begin{prop}
	\label{p.S5}
	In Lemma \ref{Section-Criterion-appendix}, $A$ is not isomorphic to the $10$-dimensional spinor variety $\bbS_{5}$.
\end{prop}

\begin{proof}
	
	As $c_2(\sE)\cdot L^8<0$, the Bogomolov inequality implies that $\sE$ is not semi-stable. Let $\sQ$ the last graded piece of the Harder-Narasimhan filtration of $\sE$ and denote by $\sG$ the quotient $\sE/\sQ$. Then the determinant $\det(\sG)$ is isomorphic to $\sO_{\bbP^{n-1}}(b)$ for some $b\leq -1$. 
	
	Firstly we assume that $\sG$ has rank $1$. Then by Lemma \ref{Lemma:Spliting-type-I} above, $b=-6$  and we have
	\[
	\sE\cong \sO_{\bbP^{10}}(3)\oplus \sO_{\bbP^{10}}(2) \oplus \sO_{\bbP^{10}}(-6).
	\]
	Let $h^*\sE\rightarrow \sO_A(2)$ be the line bundle quotient corresponding to a section $A'\subset \bbP(\sE)$ such that $\bar{h}(A')=A$. Then it is clear that we have the following factorisation
	\[
	h^*\sE\rightarrow \sO_{A}(2)\oplus \sO_A(-6) \rightarrow \sO_A(2).
	\]
	This implies that $A$ is contained in the prime divisor 
	\[
	F=\bbP(\sO_{\bbP^{10}}(2)\oplus \sO_{\bbP^{10}}(-6))\subset \bbP(\sE).
	\]
	Note that $F\cap X\rightarrow \bbP^{10}$ is a generically finite morphism of degree $2$ since $X$ is a conic bundle and $F\in |\zeta-3H|$. Nevertheless, this is impossible as $A$ is an irreducible component of $F\cap X$ and $A\rightarrow \bbP^{10}$ is of degree $e=12$.
	
	Now we assume that $\sG$ has rank $2$. Let $\sL$ be the kernel of $\sE\rightarrow \sG$. Then we have $\sL\cong \sO_{\bbP^{10}}(a)$ for some $a\geq 0$ by the construction of $\sG$. 
	
	\bigskip
	
	\textbf{Claim.} \textit{$a\leq 2$.}
	
	\bigskip
	
	\textit{Proof of Claim.} Assume to the contrary that $a>2$. By our assumption, there exists a line bundle quotient $h^*\sE\rightarrow \sO_A(2)$ with the corresponding section $A'\subset \bbP(h^*\sE)$ such that $\bar{h}(A')=A$. Moreover, as $a>2$, it follows that the composition
	\[
	h^*\sL\rightarrow h^*\sE\rightarrow \sO_A(2)
	\]
	is identically zero. Hence, we have a factorisation
	\[
	h^*\sE\rightarrow h^*\sG\rightarrow \sO_A(2).
	\]
	Let $F\subset \bbP(\sE)$ be the main component of $\bbP(\sG)\subset \bbP(\sE)$. Then $F$ is a prime divisor such that $F\in |\zeta-aH|$ and $A\subset F$. As before, the induced morphism $F\cap X\rightarrow \bbP^{10}$ is a generically finite morphism of degree $2$, while $A\rightarrow \bbP^{10}$ is of degree $12$, which is impossible. This finishes the proof of the claim.
	
	\bigskip
	
	Note that $\sG$ is semistable by our assumption. Thus the Bogomolov's inequality says that $c_2(\sG)\cdot L^8\geq 0$ (see \cite[Theorem 3.4.1]{HuybrechtsLehn2010}). Nevertheless, by the definition of Chern classes, we have
	\[
	c_2(\sG)\cdot L^8 + c_1(\sG)\cdot c_1(\sL)\cdot L^8 = c_2(\sE)\cdot L^8=-24.
	\]
	This implies 
	\[
	c_2(\sG) \cdot L^8 =-24-(-1-a)a=-24+a+a^2\leq -18,
	\]
	which is a contradiction. 
\end{proof}

\begin{remark}
	One can see that the direct sum $\sO_{\bbP^{10}}(2)\oplus \sO_{\bbP^{10}}(3) \oplus \sO_{\bbP^{10}}(-6)$ has Chern classes $(-1,-24,-36)$ with respect to $L$.
\end{remark}

Now we are in the position to prove Proposition \ref{p.rhs-conic}.

\begin{proof}[Proof of Proposition \ref{p.rhs-conic}]
	By Lemma \ref{Section-Criterion-appendix} and Proposition \ref{Prop:eleq2d-2}, the only possibilities of $A$ are as follows: a projective space, a quadric hypersurface, the Grassmann variety $\text{Gr}(2,5)$ and the $10$-dimensional spinor variety $\bbS_{10}$. If $A$ is a projective space, then it is proved in \cite{Tsukioka2006} that $A$ is a section of $f$. If $A$ is the Grassmann variety $\text{Gr}(2,5)$, then we have $e=r=5$. In particular, there does not exists positive integers $d\leq r-1$ such that $2d/e$ is an integer and we can exclude it by Lemma \ref{Section-Criterion-appendix}. The $10$-dimensional spinor variety $\bbS_{5}$ is excluded in Proposition \ref{p.S5}.
\end{proof}

\subsection{Proof of Theorem \ref{Classification-Negative-Divisors-appendix}} 

Now we are ready to prove Theorem \ref{Classification-Negative-Divisors-appendix}. We will only deal with the cases which are affected by Lemma \ref{Section-Criterion}. The proof is based on a discussion with Masaru Nagaoka.

For case (1), denote by $R_1$ and $R_2$ the extremal rays of $\NE{X}$ and, without loss of generality, we shall assume that $A\cdot R_1>0$. Then we have the following diagram
\[
\begin{tikzcd}[column sep=large, row sep=large]
	&
	X \arrow[dl,"\sigma" above] \arrow[dr,"\pi"] &
	\\
	Y &
	&
	Z
\end{tikzcd}
\]

where $\sigma$ is the contraction corresponding to $R_1$. The case affected by \cite[Lemma 3.2]{Liu2020a} is that $\sigma$ is a conic bundle and the induced morphism $A\rightarrow Y$ is not an isomorphism. Note that $Y$ is always isomorphic to the projective space $\bbP^{n-1}$ by \cite[Main Theorem]{HwangMok1999}. Thus Proposition \ref{p.rhs-conic} shows that $A$ is isomorphic to a quadric hypersurface. Then, by adjunction formula, we have 
\[
-K_X=A+(n-1)H_X,
\]
where $H_X$ is the pull-back of a hyperplane section of $Y=\bbP^{n-1}$. Consider the following short exact sequence
\[
0\rightarrow \sO_X(-A-K_X) \rightarrow \sO_X(-K_X) \rightarrow \sO_A(-K_X)\rightarrow 0.
\]
Tensoring it with $\sO_{X}((d-n+1)H_X)$ yields
\[
0\rightarrow \sO_X(dH_X) \rightarrow \sO_X(-K_X-(n-1-d)H_X) \rightarrow \sO_A(-K_A-(n-1)H_X)\rightarrow 0.
\]
Here we use the fact that $\sN_{A/X}\cong \sO_A(-dH_X)$. Moreover, as $\sO_A(-K_A)\cong \sO_A(n-1)$, pushing-forward the exact sequence by $\sigma$ yields
\[
0\rightarrow \sO_{\bbP^{n-1}}(d)\rightarrow \sE \rightarrow \sO_{\bbP^{n-1}}\oplus \sO_{\bbP^{n-1}}(-1)\rightarrow 0.
\]
This implies
\[
\sE\cong \sO_{\bbP^{n-1}}(d)\oplus \sO_{\bbP^{n-1}}\oplus \sO_{\bbP^{n-1}}(-1).
\]
Note that $-K_X-(n-1-d)H_X$ is $\sigma$-very ample, it follows that $X$ is embedded in $\bbP(\sE)$ as a divisor such that $X\in |2\zeta+aH|$ for some integer $a$, where $H$ is the pull-back of a hyperplane section of $\bbP^{n-1}$ to $\bbP(\sE)$ and $\zeta$ is the tautological divisor of $\bbP(\sE)$. Then we obtain
\begin{align*}
	-\zeta|_X - (n-1-d)H_X = K_X   & = (K_{\bbP(E)} + 2\zeta+aH)|_X \\
	& = -\zeta|_X + (d-1-n+a)H_X.
\end{align*}
Here we use the fact that $\zeta|_X=-K_X-(n-1-d)H_X$. Hence we have $a=2$. Moreover, let $F\subset \bbP(\sE)$ be the prime divisor corresponding to the quotient 
\[
\sE\rightarrow \sO_{\bbP^{n-1}}\oplus\sO_{\bbP^{n-1}}(-1).
\] 
Then $A$ is contained $X\cap F$ and we have 
\[
\sO_\bbP(\sE)(F)|_X\cong \sO_{X}(\zeta-dH)\cong \sO_X(A).
\]
Hence, we obtain $A=X\cap F$ and we are in case (1.3).

For case (2), there exists a blow-up $\sigma:X\rightarrow Y$ along a smooth centre of codimension $2$, $Y$ is a smooth Fano variety and $A\cdot R>0$, where $R$ is the extremal ray of $\NE{X}$ generated by the class of a curve contracted by $\sigma$. Moreover, there exists a Fano manifold $Z$ of dimension $n-1$, $\rho(Z)=1$ and a $\bbP^1$-bundle $\pi:Y\rightarrow Z$. Set $A_Y=\sigma(A)$. Then $A\rightarrow A_Y$ is an isomorphism and $C$ is contained $A_Y$. Denote by $d'$ the unique positive integer such that $C\in |\sO_{A_Y}(d')|$. The case affected by \cite[Lemma 3.2]{Liu2020a} is that $A_Y$ is not a nef divisor in $Y$. Then the pair $(Y,A_Y)$ is isomorphic to one of the varieties listed in (1.1) and (1.3). The case (1.1) is already done and it remains to consider the case (1.3). Nevertheless, in this case, since $Y\rightarrow Z$ is a $\bbP^1$-bundle and there exists a contraction $Y\rightarrow Z'$ sending $A_Y$ to a point, by \cite[Lemma 3.9]{CasagrandeDruel2015}, the divisor $A_Y$ must be a section of $Y\rightarrow Z$, which is a contraction. Hence, the case (1.3) does not happen.

\subsection{Some other typos}

In \cite[Proposition 2.10]{Liu2020a}, the condition "$\sL|_D$ is very ample" should be replaced by the condition "$\sL|_D$ is simply generated". Similarly, in \cite[Proposition 2.11]{Liu2020a}, the condition "$\sL|_Y$ is very ample"  should be replaced by the condition "$\sL|_Y$ is simply generated". In the proof, these two propositions are used in the case with $D$ and $Y$ being a rational homogeneous space of Picard number $1$ and it is known that any ample line bundles on rational homogeneous spaces are simply generated (see for instance \cite[Theorem 1]{RamananRamanathan1985}).

\subsection*{Acknowledgements} 
I would like to express my gratitude to Masaru Nagaoka for pointing out the mistakes in the proof of Lemma 3.2 and also for many useful discussion to correct the main theorem in this paper. I also would like to thank Jinhyung Park for pointing out the mistakes in \cite[Proposition 2.10 and 2.11]{Liu2020a}.

\def\cprime{$'$} %newcommand of prime over letters

\renewcommand\refname{Reference}
\bibliographystyle{alpha}
\bibliography{ampledivisors}

\end{document}